\documentclass[reqno, 12pt]{amsart}
\makeatletter
\let\origsection=\section \def\section{\@ifstar{\origsection*}{\mysection}} 
\def\mysection{\@startsection{section}{1}\z@{.7\linespacing\@plus\linespacing}{.5\linespacing}{\normalfont\scshape\centering\S}}
\makeatother        
\usepackage{amsmath,amssymb,amsthm}    
\usepackage{mathrsfs}    
\usepackage{mathabx}\changenotsign

\usepackage{xcolor}
\usepackage[backref]{hyperref}
\hypersetup{
    colorlinks,
    linkcolor={red!60!black},
    citecolor={green!60!black},
    urlcolor={blue!60!black}
}

\usepackage[abbrev,msc-links,backrefs]{amsrefs}

\usepackage[T1]{fontenc}
\usepackage{lmodern}

\usepackage[english]{babel}

\usepackage{geometry}
\usepackage{enumitem}
\def\rmlabel{\upshape({\itshape \roman*\,})}

\let\polishlcross=\l
\def\l{\ifmmode\ell\else\polishlcross\fi}

\let\setminus=\smallsetminus

\newtheorem{theorem}{Theorem}
\newtheorem{lemma}[theorem]{Lemma}
\newtheorem*{lemma*}{Lemma}
\newtheorem{claim}[theorem]{Claim}

\newtheorem{corollary}[theorem]{Corollary}
\newtheorem{fact}[theorem]{Fact}
\newtheorem{conj}[theorem]{Conjecture}

\theoremstyle{definition}
\newtheorem{definition}[theorem]{Definition}

\theoremstyle{remark}
\newtheorem{remark}[theorem]{Remark}

\let\theta=\vartheta
\let\rho=\varrho
\let\phi=\varphi

\def\EE{\mathbb E}
\def\NN{\mathbb N}

\def\PP{\mathbb P}

\def\RR{\mathbb R}

\def\cA{{\mathcal A}}
\def\cH{{\mathcal H}}

\def\cF{{\mathcal F}}
\def\cB{{\mathcal B}}
\def\cC{{\mathcal C}}
\def\cD{{\mathcal D}}
\def\cP{{\mathcal P}}

\def\cE{{\mathcal E}}

\def\cS{{\mathcal S}}
\def\cT{{\mathcal T}}

\DeclareMathOperator{\ee}{e}

\newcommand{\Gnp}{{G_{n,p}}}
\newcommand{\GnpOne}{{G_{n,p_1}}}
\newcommand{\GnpTwo}{{G_{n,p_2}}}
\newcommand{\seq}{\subseteq}
\renewcommand{\Pr}{\PP}
\newcommand{\prob}[1]{\Pr\left( #1 \right)}
\newcommand{\mTwo}{\ensuremath{m_2}}
\newcommand{\dTwo}{\ensuremath{d_2}}
\newcommand{\wrt}{w.r.t.\ }
\newcommand{\eqBy}[1]{\mathrel{\mathnormal{\stackrel{\eqref{#1}}{=}}}}
\newcommand{\eqByM}[1]{\mathrel{\mathnormal{\stackrel{#1}{=}}}}
\newcommand{\geBy}[1]{\mathrel{\mathnormal{\stackrel{\eqref{#1}}{\ge}}}}
\newcommand{\geByM}[1]{\mathrel{\mathnormal{\stackrel{#1}{\ge}}}}

\newcommand{\leBy}[1]{\mathrel{\mathnormal{\stackrel{\eqref{#1}}{\le}}}}
\newcommand{\leByM}[1]{\mathrel{\mathnormal{\stackrel{#1}{\le}}}}

\newcommand{\bigmid}[1][\middle]{\,#1|\,}
\newcommand{\Gbar}{\bar{G}}
\renewcommand{\(}{\left(}
\renewcommand{\)}{\right)}

\newcommand{\tn}{{\widetilde{n}}}
\newcommand{\tG}{\widetilde{G}}
\newcommand{\tcH}{\widetilde{\cH}}

\begin{document}

\title[Thresholds for asymmetric Ramsey properties]{Upper 
	bounds on probability thresholds for asymmetric Ramsey properties}

\author{Yoshiharu Kohayakawa}
\address{Instituto de Matem\'atica e Estat\'{\i}stica, Universidade de
  S\~ao Paulo, S\~ao Paulo, Brazil}  
\email{yoshi@ime.usp.br}

\author{Mathias Schacht}
\address{Fachbereich Mathematik, Universit\"at Hamburg, Hamburg, Germany}
\email{schacht@math.uni-hamburg.de}

\author{Reto Sp\"ohel}
\address{Algorithms \& Complexity Group, Max Planck Institute for Informatics, 
	Saar\-br\"ucken, Germany}
\email{rspoehel@mpi-inf.mpg.de}

\thanks{The first authors was
      partially supported by CNPq (Proc.~308509/2007-2 and
      Proc.~484154/2010-9). The collaboration
      of the first two authors was supported by a CAPESDAAD
      collaboration grant.  The second author was supported 
      through the Heisenberg-Programme of the Deutsche
      Forschungsgemeinschaft (DFG Grant SCHA 1263/4-1). The third 
      author was supported by a
      grant from the Swiss National Science Foundation and the research was
      partially carried out when the author was still at ETH Zurich.
      The authors are grateful to NUMEC/USP,
      N\'{u}cleo de Modelagem Estoc\'{a}stica e Complexidade of the
      University of S\~{a}o Paulo, for its hospitality.}

\begin{abstract} Given two graphs $G$ and $H$, we investigate for
  which functions $p=p(n)$ the random graph $\Gnp$ (the binomial
  random graph on~$n$ vertices with edge probability $p$) satisfies
  with probability $1-o(1)$ that every red-blue-coloring of its edges
  contains a red copy of $G$ or a blue copy of $H$. We prove a general
  upper bound on the threshold for this property under the assumption
  that the denser of the two graphs satisfies a certain balancedness
  condition. Our result partially confirms a conjecture by the first
  author and Kreuter, and together with earlier lower bound results
  establishes the exact order of magnitude of the threshold for the
  case in which~$G$ and $H$ are complete graphs of arbitrary size.

  In our proof we present an alternative to the so-called deletion
  method, which was introduced by R\"odl and Ruci\'{n}ski in their study
  of symmetric Ramsey properties of random graphs (i.e.\ the case
  $G=H$), and has been used in many proofs of similar results since
  then.
\end{abstract}

\maketitle

\section{Introduction} 

\subsection{Ramsey properties of random graphs}
Ramsey properties of random graphs were studied first by Frankl and
R\"odl~\cite{MR932121}, and much effort has been devoted to their
further investigation since then. Perhaps most notably, R\"odl and
Ruci{\'n}ski~\cites{MR1249720, MR1276825} established a general
threshold result that we present in the following.

For any two graphs $F$ and $H$, let  \[
  F \rightarrow (H)_k
\]
denote the property that every edge-coloring of~$F$ with~$k$ colors
contains a monochromatic copy of~$H$. Throughout, we denote the number of edges and vertices of a graph $G$
by $e_G$ and $v_G$ respectively (sometimes also by $e(G)$ and $v(G)$). We say that a graph is nonempty
if it has at least one edge.  For any graph $H$ we define
\begin{equation} \label{eq:def_d2}
d_2(H):=
\begin{cases} \displaystyle \frac{e_H-1}{v_H-2} &\text{if $v_H\geq 3$}\\
   1/2   &  \text{if $H\cong K_2$}\\
   0   &  \text{if $e_H=0$},
  \end{cases}\end{equation}
 and set
\begin{equation} \label{eq:def_m2}
m_2(H):=\max_{J\seq H}\; d_2(J)\,.
\end{equation}
We say that $H$ is \emph{$2$-balanced} if $m_2(H)=d_2(H)$, and
\emph{strictly $2$-balanced} if in addition $m_2(H)>d_2(J)$ for all
proper subgraphs $J\subsetneq H$.  With the notation above, a slightly
simplified version of the result of R\"odl and Ruci\'nski reads as
follows. (The lower bound proof given in~\cite{MR1249720} does not
cover the case where $J\seq H$ maximizing $d_2(J)$ is a triangle;
however, this case was settled earlier in~\cite{MR1182457}.)

Recall that in the binomial random graph~$\Gnp$ on~$n$ vertices, every
edge is present with probability~$0 \leq p = p(n) \leq 1$
independently of all other edges.

\begin{theorem}[R\"odl and Ruci{\'n}ski \cites{MR1249720,MR1276825}]
  \label{thm:Ramsey}
Let $k \geq 2$ and $H$ be a graph that is not a forest. Then there exist constants~$c$, $C > 0$ such that  
\[
  \lim_{n\to\infty}\prob{\Gnp \rightarrow (H)_k}
  =
  \begin{cases}
    0   &\text{if $p=p(n) \leq cn^{-1/\mTwo(H)}$}\\
    1   &\text{if $p=p(n) \geq Cn^{-1/\mTwo(H)}$},
 \end{cases}\]
where $m_2(H)$ is defined in~\eqref{eq:def_d2} and \eqref{eq:def_m2}.
\end{theorem}

We will refer to the two statements made by Theorem~\ref{thm:Ramsey} as the $0$- and the $1$-statement, respectively, and to the function $p_H(n)=n^{-1/\mTwo(H)}$ as the \emph{threshold} for the Ramsey property $F\rightarrow (H)_k$. The $1$-statement of Theorem~\ref{thm:Ramsey} is also true when $H$ is any forest that is not a matching; for the $0$-statement however there are a few well-understood nontrivial exceptions (see e.g.~\cite{MR1782847}*{Section~8.1}). 

A vertex-coloring analogue of Theorem~\ref{thm:Ramsey} was proved earlier in~\cite{MR1182457}, and generalizations of Theorem~\ref{thm:Ramsey} to the (uniform) hypergraph setting were studied in~\cites{up:frs, MR1492867, MR2318677}. Most work on the hypergraph setting has focused on the corresponding $1$-statements, i.e., on proving upper bounds on the thresholds of the respective Ramsey properties. This line of work has been settled quite recently by the results of~\cite{up:frs},  which imply $1$-statements analogous to that of Theorem~\ref{thm:Ramsey} for even more general settings. Similar results were reported by Conlon and Gowers~\cite{up:cg}.

\subsection{Asymmetric Ramsey properties}
In Theorem~\ref{thm:Ramsey} the same graph~$H$ is forbidden in every
color class. In this paper we are concerned with the natural
generalization of this setup where a different graph is forbidden in
each of the~$k$ color classes. Within classical Ramsey theory the
study of these so-called asymmetric Ramsey properties led to many
interesting questions and results; see e.g.~\cite{MR1601954}.

For any graphs $F,H_1, \dots, H_k$, let  
\[
  F \rightarrow (H_1,\dots,H_k)
\]
denote the property that every edge-coloring of~$F$ with~$k$ colors
contains a monochromatic copy of $H_i$ in color $i$ for some $1\leq
i\leq k$. 
The threshold of this asymmetric Ramsey property was determined for the case in which
all the~$H_i$ are cycles~$C_{\ell_i}$ (here $C_\ell$~denotes the cycle of length~$\ell$) by the first author and Kreuter.

\begin{theorem}[\cite{MR1609513}]
\label{thm:Ramsey asymmetric cycles} Let~$k \geq 2$ and $3 \leq \ell_1 \leq \cdots \leq \ell_k$ be integers. Then there exist constants $c$, $C > 0$ such that
\[
  \lim_{n\to\infty}
  	\prob{\Gnp \rightarrow (C_{\ell_1}, \ldots ,C_{\ell_k})}
  =
  \begin{cases}
    0   &\text{if $p=p(n) \leq cn^{-1/\mTwo(C_{\ell_2}, C_{\ell_1})}$}\\
    1   &\text{if $p=p(n) \geq Cn^{-1/\mTwo(C_{\ell_2}, C_{\ell_1})}$}\,,
  \end{cases} 
\]
where
\[ 
  \mTwo(C_{\ell_2}, C_{\ell_1}) 
  := \frac{\ell_1}{\ell_1 - 2 + (\ell_2-2)/(\ell_2-1)}
  \,.
\]
\end{theorem}

Note that the threshold does not depend on $\ell_3, \dots, \ell_k$ in
order of magnitude.

In the same paper, an explicit threshold function for asymmetric Ramsey properties involving arbitrary graphs $H_i$ is conjectured. The conjecture is stated for the two-color case, and also we will restrict our attention to this case in the following. We will briefly return to the case with more colors at the end of this paper.

For any two graphs $G$ and $H$ we let
\begin{equation} \label{eq:def_d2-asym}
  \dTwo(G, H)
  := \begin{cases}
   \displaystyle \frac{e_H}{v_H - 2 + 1/\mTwo(G)} & \text{if $e_G$,
     $e_H\geq 1$}\\ 
   0 & \text{otherwise}
   \end{cases}
  \end{equation}
(where $m_2(G)$ is defined in~\eqref{eq:def_d2} and \eqref{eq:def_m2}), 
and set
\begin{equation} \label{eq:def_m2-asym}
  \mTwo(G, H)
  := \max_{J\seq H}\; d_2(G,J)
  \,.
\end{equation}
We say that $H$ is \emph{balanced \wrt $d_2(G, \cdot)$} if $m_2(G, H)=d_2(G, H)$, and \emph{strictly balanced \wrt $d_2(G, \cdot)$} if in addition $m_2(G, H)>d_2(G,J)$ for all proper subgraphs $J\subsetneq H$.

It can be verified that $m_2(G,G)=\mTwo(G)$ for any graph $G$ and, more generally, that for any two graphs $G$ and $H$ with $\mTwo(G) \leq \mTwo(H)$ we have $m_2(G)\leq m_2(G,H)\leq m_2(H)$, with both inequalities strict if $0<\mTwo(G) < \mTwo(H)$. The conjecture in~\cite{MR1609513} is as follows.  

\begin{conj}[\cite{MR1609513}]
\label{conj:asymmetric colorings} Let~$G$ and~$H$ be graphs that are not forests with~$\mTwo(G)
\leq \mTwo(H)$. Then there exist constants $c$, $C > 0$ such that
\[
  \lim_{n\to\infty}
  	\prob{\Gnp \rightarrow (G, H)}
  =
  \begin{cases}
    0   &\text{if $p=p(n) \leq c n^{-1/{\mTwo(G, H)}}$}\\
    1   &\text{if $p=p(n) \geq C n^{-1/{\mTwo(G, H)}}$},
  \end{cases} 
\]
where $\mTwo(G, H)$ is defined in \eqref{eq:def_d2-asym} and
\eqref{eq:def_m2-asym}.
\end{conj}

The assumption that $G$ and $H$ are not forests was not made in the
original formulation of Conjecture~\ref{conj:asymmetric colorings},
but without it the $0$-statement fails to hold even in the symmetric
case, as mentioned above.

The threshold function stated in Conjecture~\ref{conj:asymmetric
  colorings} can be motivated as follows. Let $G$ and~$H$ be graphs
with $0<m_2(G)<m_2(H)$, and assume that we are looking for a
red-blue-coloring of $\Gnp$ that contains no red copy of $G$ and no
blue copy of $H$. For simplicity, suppose that
$m_2(G)=(e_G-1)/(v_G-2)$ and~$\mTwo(G,H)=e_H/(v_H-2+1/m_2(G))$.  Note
that w.l.o.g.\ we may assign color blue to all edges that are not
contained in a copy of $H$ -- in other words, only the edges of
$\Gnp$ that are contained in copies of $H$ are relevant for the Ramsey
property $\Gnp\rightarrow(G, H)$. We shall call these edges
\emph{$H$-edges} in the following.  By standard calculations, for~$p =
cn^{-1/\mTwo(G,H)}$ the expected number of $H$-edges in $\Gnp$ is of
order $n^{v_H-e_H/m_2(G,H)}=n^{2-1/m_2(G)}$, and if these edges behave
like edges of a random graph~$G_{n,p^*}$ with $p^*=c'n^{-1/m_2(G)}$,
the expected number of copies of~$G$ that are formed by such $H$-edges
and contain a given edge of $\Gnp$ is a constant depending on $c$. If
this constant is close to zero, the copies of $G$ formed by $H$-edges
in $G_{n,p}$ should be loosely scattered, and we can color one edge
blue in each of these copies without creating blue copies of $H$ in
the process. On the other hand, if this constant is large, the copies
of $G$ formed by $H$-edges of $\Gnp$ will highly intersect with each
other, and, according to the conjecture, almost surely there will
be no coloring avoiding both a red copy of~$G$ and a blue copy
of~$H$. 

The reader may wonder why a similar reasoning with the roles of $G$
and $H$ reversed is not equally justified. The reason is that whenever
$p$ is larger than $n^{-1/m_2(G)}$ by an appropriate  polylogarithmic factor
(in particular for $p =
cn^{-1/\mTwo(G,H)}$ as above), with high probability \emph{every} edge
of $\Gnp$ is contained in a copy of $G$. (Recall that $G$ is the
sparser of the two graphs.) Thus the notion of `$G$-edges' is
meaningless in our context.

A vertex-coloring analogue of Conjecture~\ref{conj:asymmetric
  colorings} was proved by Kreuter~\cite{MR1606853}. The only
significant progress towards proving Conjecture~\ref{conj:asymmetric
  colorings} since its publication in~\cite{MR1609513} concerns the
$0$-statement, which was shown to hold for the case in which~$G$
and~$H$ are complete graphs of arbitrary fixed sizes
in~\cite{MR2531778}.

The approach employed in~\cite{MR1609513} for the proof of the
$1$-statement of Theorem~\ref{thm:Ramsey asymmetric cycles} is based
on the sparse version of Szemer\'edi's regularity lemma
(see~\cites{MR1661982,MR1980964}).
The approach via sparse regularity can be extended to prove the
$1$-statement of Conjecture~\ref{conj:asymmetric colorings} for any
two graphs $G$ and $H$, provided the so-called
K{\L}R-Conjecture~\cite{MR1479298} holds for $G$ and $H$ is strictly
balanced \wrt $d_2(G,\cdot)$ (see~\cite{MR2531778};  additionally, Lemma 16 in 
\cite{MR1609513} needs to be modified slightly to relax the
condition on $H$ from `$2$-balanced' to `strictly balanced \wrt
$d_2(G,\cdot)$'). 
The K{\L}R-Conjecture has been proven for cycles of
arbitrary size, and for complete graphs on up to five vertices. For
references and a comprehensive overview of the status quo of that
conjecture, we refer to the survey article~\cite{MR2187740}.

\subsection{Our results}

In this paper we prove the $1$-statement of Conjecture~\ref{conj:asymmetric colorings} under the same balancedness assumption for $H$ as is needed for the approach via sparse regularity, but without invoking the K{\L}R-Conjecture for $G$. We say that a graph is a \emph{matching} if it has maximum degree at most $1$.

\begin{theorem}[Main result]
\label{thm:main result}
Let~$G$ and~$H$ be graphs that are not matchings such that~$H$ is
strictly balanced \wrt $d_2(G,\cdot)$.  Then there exists a constant
$C>0$ such that for $p=p(n) \geq C n^{-1/{\mTwo(G, H)}}$ we have
$$
\lim_{n\to\infty}\Pr(\Gnp \rightarrow (G, H))=1\;.
$$
\end{theorem}

Recall that we suppose that $m_2(G)\leq m_2(H)$ in
Conjecture~\ref{conj:asymmetric colorings}. One can show that the
assumption that $H$ should be strictly balanced \wrt $d_2(G,\cdot)$ in
Theorem~\ref{thm:main result} implies that $m_2(G) < m_2(H)$.  

There is an equivalent formulation of the hypothesis of strict
balancedness in Theorem~\ref{thm:main result}.  For every subgraph~$J$
of~$H$, let~$\mu(J;n,p)$ be the expected number of occurences of~$J$
in~$G_{n,p}$.  Then~$H$ is strictly balanced \wrt $d_2(G,\cdot)$ if
and only if
\[\mu(H;n,n^{-1/{\mTwo(G, H)}}) = o( \mu(J;n,n^{-1/{\mTwo(G,H)}}))\]
for every proper subgraph~$J$ of~$H$ (see Remark~\ref{rem:least-frequent} and
Lemma~\ref{lemma:equivalence}(\textit{ii}) below). 

Our proof of Theorem~\ref{thm:main result} does not use sparse
regularity at all, and has in fact more in common with the original
proof of the $1$-statement for the symmetric case
(Theorem~\ref{thm:Ramsey}), due to R\"odl and Ruci\'nski, than with
the proof of Theorem~\ref{thm:Ramsey asymmetric cycles} given
in~\cite{MR1609513}.  We believe that a feature of interest in our
proof is that it introduces a different approach for handling certain
technical difficulties that are dealt with in the R\"odl--Ruci\'nski
proof via the so called `deletion method' (for details, see
Section~\ref{sec:deletion}).

Together with the lower bound results for complete graphs we already mentioned~\cite{MR2531778}, our result establishes general threshold functions for the case where $G=K_\ell$ and $H=K_r$ are complete graphs of fixed sizes $\ell<r$.

\begin{corollary}
\label{cor:cliques}
Let $3\leq \ell < r$ be integers. Then there exist constants $c$, $C>0$ such that
\[
  \lim_{n\to\infty}
  	\prob{\Gnp \rightarrow (K_\ell, K_r)}
  =
  \begin{cases}
    0   &\text{if $p=p(n) \leq  c n^{-1/{\mTwo(K_\ell, K_r)}}$}\\
    1   &\text{if $p=p(n) \geq  C n^{-1/{\mTwo(K_\ell, K_r)}}$},
  \end{cases} 
\]
where
\begin{equation*} 
  \mTwo(K_\ell, K_r)
  = \frac{\binom{r}{2}}{r - 2 + 2/(\ell+1)}
  \,.
\end{equation*}
\end{corollary}

We can use Theorem~\ref{thm:main result} to infer statements about the
existence of locally sparse graphs~$F$ that enjoy the asymmetric
Ramsey property $F\rightarrow(G,H)$, similarly to those presented
in~\cite{MR1276825} for symmetric Ramsey properties. We refrain from a
general statement of these results, and only mention the following
corollary, which is an asymmetric variant of Corollary~5
in~\cite{MR1276825} and can be deduced analogously.

\begin{corollary}
  \label{cor:local_sparse}
  For all $3\leq \ell\leq r$, there exists a constant $C$ such for
  $m\geq Cn^{2-1/m_2(K_\ell, K_r)}$, almost all graphs $F$ on $n$
  vertices with $m$ edges that contain no copy of $K_{r+1}$ satisfy
  $F\rightarrow(K_\ell,K_r)$.
\end{corollary}

We close with a deterministic consequence of
Corollary~\ref{cor:cliques}.  A graph $F$ is called
\emph{Ramsey-critical}, or simply \emph{critical}, for a pair of
graphs $(G,H)$, if $F\rightarrow (G,H)$ but, for any proper subgraph
$F'$ of $F$, the relation $F'\rightarrow (G,H)$ fails. The
pair~$(G,H)$ is called \emph{Ramsey-finite} if the class $\cC(G,H)$ of
all graphs that are critical for $(G,H)$ is finite, and
\emph{Ramsey-infinite} otherwise. Note that, by definition, the Ramsey
property $F\rightarrow (G,H)$ is equivalent to $F$ containing a copy of a
graph from~$\cC(G,H)$.

The following result was originally proved by constructive means by Burr,  Erd\H{o}s, and Lovasz~\cite{MR0419285}. We obtain an alternative (non-constructive) proof as an immediate consequence of  Corollary~\ref{cor:cliques}.

\begin{corollary}
  \label{cor:R-inf}
For all $3\leq \ell< r$, the pair $(K_\ell, K_r)$ is Ramsey-infinite.
\end{corollary}

\begin{proof}
  It is well-known (and can be shown similarly 
  to~\cite{MR1782847}*{Theorem~3.9}) that for any finite family~$\cF$, the property
  $\cP=\cP(\cF)$ of containing a copy of a graph from~$\cF$ admits a
  threshold~$p_\cF=p_\cF(n)$ such that,
  for any constant $c>0$, the random graph $\Gnp$ with $p=cp_\cF$
  satisfies 
  $$
  0<\liminf_{n\to\infty} \prob{\Gnp \in \cP}\leq
  \limsup_{n\to\infty} \prob{\Gnp \in
    \cP}<1\,.
  $$ 
  Corollary~\ref{cor:cliques} implies that the property $F
  \rightarrow (K_\ell, K_r)$ does not admit a function~$p_\cF=p_\cF(n)$
  and hence Corollary~\ref{cor:R-inf} follows.
\end{proof}

\subsection{An alternative to the deletion method}
\label{sec:deletion}

Our proof of Theorem~\ref{thm:main result} reuses many ideas of the
proof of the $1$-statement of Theorem~\ref{thm:Ramsey} given
in~\cite{MR1276825}. However, we point out one
particular technical issue that is solved in a quite different way
from~\cite{MR1276825} in our proof. Namely, at some point in the proof
one needs to control the upper tail of the random variable that
counts the number of copies of some given graph $T$ in $\Gnp$.

In typical proofs of similar results (see, e.g., \cites{up:frs, MR1492867,
  MR2318677, up:sch}), this is taken care of by the so-called
\textit{deletion method} (see also \cite{MR2096818}), i.e., by allowing the deletion of a small
fraction of edges to get the desired exponentially small error
probability. This is formalized in the `deletion lemma' 
\cite{MR1276825}*{Lemma~4} (see also \cite{MR1782847}*{Lemma~2.51}).

This deletion lemma is then combined with a `robustness lemma'
\cite{MR1276825}*{Lemma~3} (see also~\cite{MR1782847}*{Lemma 2.52}),
which states that monotone properties (like the Ramsey properties
discussed here) that hold with probability exponentially close to $1$
continue to hold with similarly high probabilities even if an
adversary is allowed to delete a small fraction of the edges. This robustness
lemma is needed to guarantee that the few edges that were deleted to
control the number of copies of $T$ do not destroy other properties
that are important for the proof.

In our proof we use a different and arguably simpler approach to
control the number of copies of~$T$. Namely, we \emph{condition} on
the number of copies of~$T$ in~$\Gnp$ not being too large, and apply
the Harris inequality \cite{MR0115221} (Theorem~\ref{thm:fkg})
to show that this only \emph{increases} the probability that other
relevant properties fail to hold  (and, hence, bounding the
probability of such bad events in the conditional space from
above gives upper bounds for the probability of those bad
events in the original space). Thus we may work in the conditional
space. The fact that the event on which we condition holds with
reasonable probability (constant probability is more than enough here)
implies that the conditional space we are considering behaves
essentially like the original space, \emph{except that with
  probability~$1$ the number of copies of~$T$ is not too large.} Thus
there is no need to delete edges in our approach.  We believe that
many of the earlier proofs in the field, in particular the proof given
in \cite{MR1276825} for the symmetric case (Theorem~\ref{thm:Ramsey}),
can be simplified analogously from the technical point of view.

\subsection{Organization of this paper}
We collect a number of definitions and auxiliary statements in
Section~\ref{sec:preliminaries}, and prove Theorem~\ref{thm:main
  result} in Section~\ref{sec:proof}.  We discuss possible extensions
of our results in Section~\ref{sec:concluding}.

\section{Preliminaries}
\label{sec:preliminaries}

\subsection{Basic inequalities}
We begin by stating some equalities that follow immediately from the
definitions of $m_2(G)$ and $m_2(G,H)$, and that will be used
throughout this paper.  Recall that we call a graph \emph{nonempty} if
it has at least one edge. The definitions in~\eqref{eq:def_d2} and
\eqref{eq:def_m2} imply that for any nonempty graph $G$ and any
subgraph $I\seq G$ with $v_I\geq 2$ we have
\begin{equation} \label{eq:exponent-0}
 v_I - \frac{e_I}{m_2(G)} \geq 2-\frac{1}{m_2(G)}
\end{equation}
(with equality for $I=G$ if~$G$ is 2-balanced).
Similarly, the definitions in~\eqref{eq:def_d2-asym} and \eqref{eq:def_m2-asym} imply that for any two nonempty graphs $G$ and $H$ and any subgraph $J\seq H$ with $v_J\geq 2$ we have
\begin{equation} \label{eq:exponent-1}
 v_J-\frac{e_J}{m_2(G,H)} \geq 2-\frac{1}{m_2(G)}
\end{equation}
(with equality for $J=H$ if~$H$ is balanced w.r.t.\ $d_2(G, \cdot)$). Combining the previous two equalities yields in particular that for any two nonempty graphs $G$ and $H$ we have
\begin{equation}
  \label{eq:exponent-2}
  v_G-2+(e_G-1)(v_H-2) -  \frac{(e_G-1)e_H}{m_2(G,H)}
  \geBy{eq:exponent-1} v_G-2- \frac{e_G-1}{m_2(G)}\geBy{eq:exponent-0}
  0 \,,
\end{equation}
which will become important later on.

\subsection{\texorpdfstring{$H$-covered copies}{H-covered copies}}

The following definitions will be crucial in our inductive scheme.

\begin{definition} \label{def:H-covered}
For graphs $H$ and $A$, we denote by $E_H(A)\seq E(A)$ the union of the edge
sets of all copies of $H$ in $A$. We will refer to the edges in $E_H(A)$ as the \emph{$H$-edges} of $A$. Furthermore, we say that a copy $\Gbar$ of a graph $G$ in $E_H(A)$ is \emph{$H$-covered} in $A$ if there is a family of $e_G$ pairwise edge-disjoint copies of $H$ in $A$ such that each edge of $\Gbar$ is contained in (exactly) one of these copies.
\end{definition}

Note that \emph{not} every copy of $G$ that is formed by $H$-edges of $A$ is $H$-covered in $A$.

\begin{definition} \label{def:FGH}
For any two graphs $G$ and $H$, let  $\cF(G,H)$ denote the family of all graphs obtained by taking a copy of $G$ and embedding each of its edges into a copy of $H$ such that these $e_G$ copies of $H$ are pairwise edge-disjoint (not nessarily vertex-disjoint). 
\end{definition}

We denote the graphs in $\cF(G,H)$ by $G^H$, and refer to a copy of $G$ in $G^H$ that can be used to construct $G^H$ as described as a \textit{central copy} of $G$ in $G^H$ (in general, for a given $G^H \in \cF(G,H)$ such a central copy is not uniquely defined). Note that a copy of $G$ in some graph $A$ is $H$-covered if and only if it is a central copy in a copy of some graph $G^H \in \cF(G,H)$ in $A$. 

For any~$G$ and $H$ and any graph $G^H\in\cF(G,H)$, let
$$
L(G^H):= v_G + e_G\cdot(v_H-2) - v(G^H)\geq 0\,.
$$
Intuitively, this quantity denotes the number of vertices that are `lost' because the
copies of $H$ forming~$G^H$ intersect in more vertices than specified by $G$. Thus we have
\begin{equation} \label{eq:evGH}
\begin{split}
e(G^H)&=e_G\cdot e_H \,,\\
v(G^H)&= v_G+e_G\cdot(v_H-2)-L(G^H)\,.
\end{split}
\end{equation}

Our induction is on the number of edges of~$G$, and we will mostly
need the above definitions for a certain graph $G_-$ with $e(G)-1$
edges to which we apply the induction hypothesis. The following
technical lemma will become important later on.

\begin{lemma} \label{lemma:goalJ}
Let $G$ be a graph that is not a matching, let $H$ be a nonempty graph, and fix some subgraph $G_-\subseteq G$ with $e(G_-)=e(G)-1$ and $v(G_-)=v(G)$. Furthermore, let a graph $G_-^H\in\cF(G_-,H)$ with central copy $G'_-$ be given, and let $g$ denote a vertex pair that completes $G'_-$ to a copy of $G$ when inserted as an edge. Then every subgraph $J\seq G^H_{-}$ that contains the two vertices of $g$ satisfies
\begin{equation*}
  v(J)-\frac {e(J)}{m_2(G,H)} \geq 2-L(G^H_{-})\,.
\end{equation*}
\end{lemma}

\begin{proof}
Note that it suffices to prove the claim for induced subgraphs $J\seq G^H_{-}$. We consider a fixed such subgraph $J$ and decompose it as follows. Let $E':=
E(G'_{-})$ denote the edge set of the central copy $G'_-$. For $f\in E'$, let $J_f$ denote the intersection of $J$
with the corresponding copy of $H$ in $G^H_{-}$ (the graph $J_f$ may contain isolated vertices). Furthermore,
let $I_0$ denote the intersection of $J$ with $G'_{-}$, and set
$V_0:=V(I_0)=V(J)\cap V(G'_-)$, $E_0:=E(I_0)=E(J)\cap E(G'_-)$. Observe that the assumption that $J$ is an induced subgraph of $G^H_{-}$ implies that also $I_0$ is an induced subgraph of $G'_-$. Furthermore, due to our assumption that $J$ contains the two vertices of $g$, also $I_0$ contains the two vertices of $g$.

Note that
\begin{equation}\label{eq:eJ}
\begin{split}
e(J)&= \sum_{\substack{f\in E'\\v(J_f)\geq 2}} e(J_f) 
\end{split}
\end{equation}
and
\begin{equation}\label{eq:vJ}
\begin{split}
v(J)& \geq v(I_0) + \sum_{\substack{f \in E'}} \big(v(J_f) - |f\cap V_0|\big) - L(G^H_{-}) \\
& \geq v(I_0) + \sum_{\substack{f\in E'\\v(J_f)\geq 2}} \big(v(J_f) - |f\cap V_0|\big) - L(G^H_{-})\,, 
\end{split}
\end{equation}
where the first inequality is due to the fact that the big sum
overcounts the actual number of vertices of $J$ by at most $L(G^H_{-})$
(i.e., $J\seq G_{-}^H$ `loses' at most as many vertices as $G_{-}^H$
because of vertex-overlapping copies of~$H$).

Combining~\eqref{eq:eJ} and~\eqref{eq:vJ} yields that
\begin{align*} 
	v(J)-\frac{e(J)}{m_2(G,H)} 
	& \geq v(I_0) + \sum_{\substack{f\in E'\\v(J_f)\geq 2}} T\(v(J_f)-|f\cap V_0| - \frac{e(J_f)}{m_2(G,H)}\) - L(G^H_{-})\\
	& \geBy{eq:exponent-1} 
	v(I_0) + \sum_{\substack{f\in E'\\v(J_f)\geq 2}} \(2 - |f\cap V_0| - \frac{1}{m_2(G)}\)- L(G^H_{-})\\ 
	&= 
	v(I_0) - \frac{e(I_0)}{m_2(G)} + \sum_{\substack{f\in E'\setminus E_0\\v(J_f)\geq 2\\ |f\cap V_0|\leq 1 }} \(2 - |f\cap V_0| - \frac{1}{m_2(G)}\)- L(G^H_{-}) \,,
\end{align*}
where for the equality we used that the edges $f\in E'=E(G'_-)$ with $|f\cap V_0|=2$ are exactly the edges in $E_0=E(I_0)$ due to the fact that $I_0$ is an induced subgraph of $G'_-$. Using that $m_2(G)\geq 1$, we may omit the remaining sum, and observing that adding the edge $g$ to $I_0$ yields a graph $I^+_0$ that is isomorphic to a subgraph of $G$, we obtain further
\begin{equation*} 
\begin{split}
v(J) - \frac{e(J)}{m_2(G,H)} 
 &\geq v(I^+_0) - \frac{e(I^+_0)-1}{m_2(G)} - L(G_{_-}^H)\\
 &\geBy{eq:exponent-0}  2 - L(G_{_-}^H)\,, 
\end{split}
\end{equation*}
concluding the proof of Lemma \ref{lemma:goalJ}.
\end{proof}

\subsection{\texorpdfstring{The parameters $m^*(H)$ and $x^*(H)$}{The parameters m*(H) and x*(H)}}
In this section we introduce two graph parameters~$m^*(H)$ and
$x^*(H)$ that will play in important role in our proof. The parameter~$m^*(H)$ is a convenient quantity to capture the concept of $H$ being
`its own least frequent subgraph' that many authors have used before
(see Remark~\ref{rem:least-frequent} below). The parameter $x^*(H)$ is
a rescaled version of $m^*(H)$ that is tailored to the specifics of
the problem studied in this paper.

\begin{definition} \label{def:m-star}
For any graph $H$ with $v_H\geq 3$, let
\begin{equation} \label{eq:def_mstar}
m^*(H):= \min_{\substack{J\seq H:\\2\leq v_J<v_H}} \frac{e_H-e_J}{v_H-v_J}
\end{equation}
and, if $H$ is nonempty,
\begin{equation} \label{eq:def_x-star}
x^*(H):=\frac{m^*(H)}{e_H-m^*(H)(v_H-2)}\,.
\end{equation}  
\end{definition}

Note that for any graph $H$, the parameter $m^*(H)$ is nonnegative, and that $m^*(H)=0$ if and only if $H$ has an isolated vertex.
It follows from~\eqref{eq:def_mstar} that for any nonempty graph $H$ with $v_H\geq 3$ the parameter $x^*(H)$ as defined in~\eqref{eq:def_x-star} is well-defined and positive. Note that solving~\eqref{eq:def_x-star} for $m^*(H)$ yields
\begin{equation} \label{eq:x-star-m-star}
 m^*(H)=\frac{e_H}{v_H-2+1/x^*(H)}\,,
\end{equation}
which connects $m^*(H)$ to $d_2(G,H)$ as defined in~\eqref{eq:def_d2-asym}. More specifically, the point here is that comparing $m^*(H)$ to $d_2(G,H)$ can be formulated equivalently as comparing  $x^*(H)$ to $m_2(G)$.
\begin{remark} \label{rem:least-frequent}
It follows from the definition of $m^*(H)$ in~\eqref{def:m-star} that
\begin{equation*} v_J - \frac{e_J}{m^*(H)} \geq v_H-\frac{e_H}{m^*(H)}\,,
\end{equation*}
for all subgraphs $J\seq H$ with $v_J\geq 2$. Thus for $p\leq n^{-1/m^*(H)}$ we have
$n^{v_J}p^{e_J}\geq n^{v_H}p^{e_H}$ for all such $J$, which means that the expected number of copies of $H$ in $\Gnp$ does not exceed the expected number of copies of any subgraph $J\seq H$ with $v_J\geq 2$ by more than a constant factor.
\end{remark}

In some sense, both $m^*(H)$ and $x^*(H)$ measure `how balanced' $H$
is.  Below we will prove some general results that make this
precise. These will in particular imply the following lemma, which
restates the hypothesis of Theorem~\ref{thm:main result} in two
alternative forms that are more convenient for us.

\begin{lemma} \label{lemma:equivalence}
For any two nonempty graphs $G$ and $H$ with $v_H\geq 3$, the following statements are equivalent.
\begin{enumerate}[label=\rmlabel]
 \item\label{it:equivi} $H$ is strictly balanced \wrt $d_2(G, \cdot)$,
 \item\label{it:equivii} $m_2(G,H)<m^*(H)$,
 \item\label{it:equiviii} $m_2(G)< x^*(H)$.
\end{enumerate}
\end{lemma}

Lemma~\ref{lemma:equivalence} will be proved in
Section~\ref{sec:2.3.2} below.

\subsubsection{The parameter $m^*(H)$ and general density measures}
\label{sec:2.3.1}

For arbitrary (possibly negative) values $a\leq 1$ and $b < 2$, we
define for any graph $H$ the density measure
\begin{equation} \label{eq:def_dpq}
  d_{a,b}(H)
  := \begin{cases}
   \displaystyle \frac{e_H-a}{v_H-b} & \text{if $e_H\geq 1$}\\
   0 & \text{otherwise},
   \end{cases}
\end{equation}
and set
$$m_{a,b}(H):=\max_{J\seq H} d_{a,b}(J)\,.$$
As usual we say that $H$ is \emph{balanced \wrt $d_{a,b}$} if $m_{a,b}(H)=d_{a,b}(H)$, and \emph{strictly balanced \wrt $d_{a,b}$} if in addition $m_{a,b}(H)>d_{a,b}(J)$ for all proper subgraphs $J\subsetneq H$.

\begin{lemma} \label{lemma:mpq} Let $d_{a,b}$ be a density measure as
  in~\eqref{eq:def_dpq}.  A graph $H$ with $v_H\geq 3$ is balanced
  w.r.t.~$d_{a,b}$ if and only if $m^*(H)\geq d_{a,b}(H)$ (or, equivalently, if and only if $m^*(H)\geq m_{a,b}(H)$).  Similarly,
  a graph $H$ with $v_H\geq 3$ is strictly balanced w.r.t.~$d_{a,b}$
  if and only if $m^*(H) > d_{a,b}(H)$ (or, equivalently, if and only if $m^*(H) > m_{a,b}(H)$).
\end{lemma}

For the proof we use the following elementary observation, which we
state separately for further reference.

\begin{fact}
  \label{fact:ratios}
  For $a$, $c\in\RR$ and $b > d > 0$, we
  have
  $$ \quad \frac{c}{d}\leq \frac{a}{b}\,\Longleftrightarrow \quad \frac{a-c}{b-d}\geq \frac{a}{b}$$
  and, similarly,
  $$ \quad \frac{c}{d}< \frac{a}{b}\,\Longleftrightarrow \quad
  \frac{a-c}{b-d}> \frac{a}{b}\,.$$ 
\end{fact}

\begin{proof}[Proof of Lemma~\ref{lemma:mpq}] 
Observe that $H$ is balanced w.r.t. $d_{a,b}$ if and only if for all subgraphs $J\seq H$ with $v_J\geq 2$ we have
$$\frac{e_J-a}{v_J-b}\leq \frac{e_H-a}{v_H-b}$$
(note that this condition is always satisfied for graphs $J$ with $v_J\geq 2$ and $e_J=0$).
By Fact~\ref{fact:ratios}, this is equivalent to the requirement that
\begin{equation*}
	\frac{e_H-e_J}{v_H-v_J}
	=  \frac{(e_H-a)-(e_J-a)}{(v_H-b)-(v_J-b)}\stackrel{\text{Fact~\ref{fact:ratios}}}{\geq} \frac{e_H-a}{v_H-b}
\end{equation*}
for all subgraphs $J\seq H$ with $2\leq v_J< v_H$, i.e., to $m^*(H)\geq d_{a,b}(H)$.

The statement for `strictly balanced' follows analogously using the second statement of Fact~\ref{fact:ratios}.
\end{proof}

\subsubsection{The parameter $x^*(H)$ and the asymmetric $2$-density} 
\label{sec:2.3.2}

As a consequence of Lemma~\ref{lemma:mpq} we obtain the next lemma, which is specifically concerned with the asymmetric $2$-density.

\begin{lemma} \label{lemma:balanced} Let $G$ be a nonempty graph.  A
  nonempty graph $H$ with $v_H\geq 3$ is balanced \wrt $d_2(G,\cdot)$
  if and only if $m_2(G)\leq x^*(H)$.  Similarly, a nonempty graph $H$
  with $v_H\geq 3$ is strictly balanced \wrt $d_2(G,\cdot)$ if and
  only if $m_2(G)< x^*(H)$.
\end{lemma}

\begin{proof} For any $x> 0$ and any graph $H$ with $v_H\geq 2$, let 
\begin{equation}\label{eq:def_d2xH}
  d_2(x,H)
  := \begin{cases}
   \displaystyle \frac{e_H}{v_H-2+1/x} & \text{if $e_H\geq 1$}\\
   0 & \text{otherwise},
   \end{cases}
\end{equation}
and set $$m_2(x,H):=\max_{\substack{J\seq H}}
d_2(x,J)\,.$$

Note that $d_2(m_2(G),H)$ as defined in \eqref{eq:def_d2xH} coincides with~$d_2(G,H)$ as defined in~\eqref{eq:def_d2-asym}, and that,  according to~\eqref{eq:x-star-m-star}, for any nonempty graph $H$ with $v_H\geq 3$ we have
\begin{equation} \label{eq:x-m}
 m^*(H) = d_2(x^*(H),H)\,.
\end{equation}  
With Lemma~\ref{lemma:mpq} we obtain that $H$ is balanced \wrt $d_2(G,\cdot)$ if and only if $$d_2(x^*(H),H)\eqBy{eq:x-m} m^*(H)\geByM{\text{L.~\ref{lemma:mpq}}} d_2(G,H) =  d_2(m_2(G),H)\,.$$
 Since $d_2(x,H)$ is monotone increasing in $x$, this is equivalent to $m_2(G)\leq x^*(H)$, as claimed.

The statement for `strictly balanced' follows analogously using the second statement of Lemma~\ref{lemma:mpq}.
\end{proof}

\begin{proof}[Proof of Lemma~\ref{lemma:equivalence}]
The equivalence of~\ref{it:equivi} and~\ref{it:equivii}
 follows from Lemma~\ref{lemma:mpq}, observing that, for any
fixed nonempty graph $G$, the parameter $d_2(G,H)$ defined
in~\eqref{eq:def_d2-asym} is a density measure as
in~\eqref{eq:def_dpq} (with $a=0$ and $b=2-1/m_2(G)$).

The equivalence of~\ref{it:equivi} and~\ref{it:equiviii} is stated in Lemma~\ref{lemma:balanced}.
\end{proof}

\subsection{Other preliminaries} As already mentioned, we will make crucial use of the Harris inequality~\cite{MR0115221} (which also arises as a special case of the FKG inequality~\cite{MR0309498} and various other related inequalities). 

Throughout, we will assume that the random graph $\Gnp$ is generated
on the vertex set $[n]=\{1,\dots, n\}$.  For the purposes of this
paper, a graph property is a family of \emph{labelled} graphs on the vertex set $[n]$
(which is not necessarily closed under isomorphism), where~$n$ will be
clear from the context. We say that a
graph property $\cA$ is \emph{decreasing} if for any two graphs $G$
and $H$ on vertex set $[n]$ the following holds: if $G\in\cA$ and $H\seq G$, we
also have~$H\in\cA$. Similarly, we say that a graph property $\cA$
is \emph{increasing} if for any two graphs~$G$ and $H$ on vertex set $[n]$ the
following holds: if $G\in\cA$ and $H\supseteq G$, we also have~$H\in\cA$. Note that the complement of a decreasing property is
increasing, and vice versa.

\begin{theorem} [Harris~\cite{MR0115221}] \label{thm:fkg} For any two
  decreasing (increasing) graph properties $\cA$ and $\cB$ and any
  $n\in\NN$ and $0\leq p\leq 1$, we have
  $$
  \Pr(\Gnp\in\cA\cap \cB)\geq\Pr(\Gnp\in\cA)\Pr(\Gnp\in\cB)\,,
  $$
  or, equivalently if $\Pr(\Gnp\in\cB)>0$,
  $$
  \Pr(\Gnp\in\cA\,|\,\Gnp\in\,\cB)\geq \Pr(\Gnp\in\cA)\,.
  $$
\end{theorem}

Clearly, it follows from Theorem~\ref{thm:fkg} that, for the binomial
random graph~$\Gnp$ the probability of any decreasing (respectively,
increasing) event~$\cA$ does not decrease if we condition on another
decreasing (respectively, increasing) event~$\cB$.

Janson's inequality is a very useful tool in probabilistic
combinatorics. In many cases, it yields an exponential bound on lower
tails where the second moment method only gives a considerably weaker
bound. Here we formulate a version tailored to random graphs.
 
\begin{theorem}[Janson~\cite{MR1138428}]\label{thm:janson}
  Consider a family~$\cH=\{H_i \mid i\in I\}$ of subgraphs of the
  complete graph on the vertex set $[n]$. For each~$H_i \in \cH$, let~$X_i$ denote the
  indicator random variable for the event~$H_i \seq \Gnp$, and, for
  each ordered pair~$(H_i,H_j)\in\cH\times\cH$ with $i \neq j$, write~$H_i
  \sim H_j$ if~$H_i$ and~$H_j$ are not edge-disjoint. Let
\begin{eqnarray*}
  X &=& \sum_{H_i \in \cH} X_i \,,\\
  \mu &=& \EE[X] = \sum_{H_i \in \cH}p^{e(H_i)} \,,\\
  \Delta &=& \sum_{\substack{(H_i,H_j)\in\cH\times\cH \\ H_i \sim H_j}}
    \EE[X_{i}  X_{j}] = \sum_{\substack{(H_i,H_j)\in\cH\times\cH \\ H_i \sim H_j}} p^{e(H_i) + e(H_j) - e(H_i \cap H_j)}\,.
\end{eqnarray*}
Then for all~$0 \leq \delta \leq 1$ we have
\[
  \Pr(X \leq (1 - \delta)\mu) 
  \leq \ee^{-\frac{\delta^2 \mu^2}{2(\mu + \Delta)}} \,.
\]
\end{theorem}

Often Janson's inequality is applied with $\cH$ being the family of all copies of some given fixed graph $H$ in the complete graph $K_n$. The concept of $(\rho,d)$-denseness will allow us to derive very similar results when applying Janson's inequality  with $\cH$ being the family of all copies of $H$ in a graph $F\seq K_n$ that is not necessarily complete.

\begin{definition}
  \label{def:rho-d-dense} For any $\rho>0$ and $0< d\leq 1$, 
  a graph $F$ on vertex set $[n]$ is said to be \emph{$(\rho,d)$-dense}
  if for every subset $V\seq [n]$ with $|V|\geq \rho n$ we have
  $$
  e(F[V])\geq d \binom{n}{2}\,,
  $$  
  where $F[V]$ denotes the subgraph induced by $F$ on $V$.
\end{definition}

We will use the following fact (for a proof see e.g.~\cite{MR1276825}).
\begin{lemma}
  \label{lemma:rho-d-dense}
  For all $0<d\leq 1$ and $\ell\geq 1$, there exist positive constants $\rho$, $n_0$ and $c_0$ such that every $(\rho, d)$-dense graph on $n \geq n_0$ vertices contains
  at least $c_0 n^\ell$ complete subgraphs $K_\ell$.
\end{lemma}

\subsection{Edge-disjoint copies} The tools and definitions presented in the previous section come together in the following technical lemma, which states that under the appropriate assumptions, a random subgraph of a $(\rho,d)$-dense graph contains a large family of pairwise edge-disjoint copies of a given graph $H$. The key idea of applying Tur\'{a}n's Theorem to a suitably defined auxiliary graph is due to Kreuter~\cite{MR1606853}. We will use Tur\'{a}n's Theorem in the following form (see e.g.~\cite{MR0384579}*{p.~282}). 
  
\begin{theorem}[Tur\'an] \label{thm:turan}
Let $G$ be a graph. Then $G$ has an independent set of size at least
$$\frac{v(G)^2}{v(G)+2e(G)}\,.$$
\end{theorem}

\begin{lemma} \label{lemma:disjoint} Let $H$ be a nonempty graph with
  $v_H\geq 3$. For any $0< d\leq 1$, there exist positive constants
  $\rho$, $n_0$ and $b$ such that for $n\geq n_0$ and $p\leq
  n^{-1/m^*(H)}$ the following holds: If $F\seq K_n$ is a
  $(\rho,d)$-dense graph on $n$ vertices, then, with probability at
  least $1-2^{-bn^{v_H}p^{e_H}+1}$, the graph $F\cap\Gnp$ contains a
  family of at least $bn^{v_H}p^{e_H}$ pairwise edge-disjoint copies
  of~$H$.
\end{lemma}

Note that for $p=o(n^{-1/m^*(H)})$ we have $n^{v_H}p^{e_H} = o(n^2 p)$ (recall Remark~\ref{rem:least-frequent}), so the error probability in Lemma~\ref{lemma:disjoint} is not as high as it may look like at first glance.

\begin{proof}[Proof of Lemma~\ref{lemma:disjoint}]
Let
\begin{equation} \label{eq:def_rhoprime-disjoint}
\rho:=\rho(v_H,d)\,, \qquad n_0:=n_0(v_H,d)\,, \qquad  c_0:=c_0(v_H,d) \leq 1
\end{equation}
denote the constants obtained by applying Lemma~\ref{lemma:rho-d-dense} with $\ell:= v_H$ and $d$. We shall prove Lemma~\ref{lemma:disjoint} for $\rho$ and $n_0$ as defined in \eqref{eq:def_rhoprime-disjoint} and
\begin{equation} \label{eq:def_b-disjoint}
  b:=\frac{c_0^2}{16^{v_H^2+1}}\,. 
\end{equation}

Let $F$ be a $(\rho,d)$-dense graph on $n\geq n_0$ vertices be given, and set
\begin{equation} 
 \cA:=\left\{K\seq K_n \bigmid \parbox{0.55\displaywidth}{$F\cap K$ contains a family of at least $bn^{v_H}p^{e_H}$ pairwise edge-disjoint copies of $H$}\right\} \,.
\end{equation}
Note that $\cA$ is increasing.
Our goal is to bound $\Pr(\Gnp\in\neg\cA)$ from above.

Denote by $\cH$ the family of all copies of $H$ in $F$.
  By our choice of constants in~\eqref{eq:def_rhoprime-disjoint}, the assumption that $F$ is $(\rho,
  d)$-dense yields with  Lemma~\ref{lemma:rho-d-dense} that there are at least $c_0
  n^{v_H}$ complete graphs of order $v_H$ in $F$. In
  particular, we have
\begin{equation} \label{eq:cH-disjoint}
 |\cH|\geq c_0 n^{v_H}\,.
\end{equation}

We will apply Janson's inequality (Theorem~\ref{thm:janson}) to the
family $\cH$.  For any graph $K\seq K_n$, we let $\cH(K)\seq \cH$ 
denote the family of all copies of $H$ in $F \cap K$. We obtain for $\mu$ as defined in Theorem~\ref{thm:janson} that
\begin{equation} \label{eq:mu-disjoint}
 \mu = \EE[|\cH(\Gnp)|] = |\cH|\cdot p^{e_H}\geBy{eq:cH-disjoint} c_0 n^{v_H}p^{e_H}\,.
\end{equation}

Let $\cS$ be the family of all pairwise nonisomorphic graphs that are
unions of two copies of~$H$ that intersect in at least one edge.  For
a fixed graph $S\in\cS$, let $J$ denote the intersection of the two
copies of $H$. Owing to the assumption that $p\leq n^{-1/m^*(H)}$, we obtain for any nonempty subgraph $J\seq H$ that
\begin{equation} \label{eq:H-rarest-disjoint}
 n^{v_H-v_J}p^{e_H-e_J} 
 \leq n^{v_H-v_J-(e_H-e_J)/m^*(H)}
 \leBy{eq:def_mstar}1\,, 
\end{equation}
which implies that for any $S\in \cS$ we have
\begin{equation} \label{eq:exp_S-disjoint}
  n^{v_S}p^{e_S} =  n^{2v_H-v_J}p^{2e_H-e_J} \leBy{eq:H-rarest-disjoint} n^{v_H}p^{e_H}\,.
\end{equation}
As there are at most $n^{v_S}$ copies of $S$ in $F$, and since
each such copy corresponds to at most $((v_S)_{v_H})^2\leq (2v_H)^{2v_H} \leq 4^{v_H^2}$
pairs $(H_i, H_j)\in\cH\times \cH$, $i\neq j$ with $H_i\cup H_j\cong S$, we obtain for $\Delta$ as defined in Theorem~\ref{thm:janson} that
 \begin{equation} \label{eq:Delta-disjoint}
\begin{split}
  \Delta &= \sum_{S\in\cS} \sum_{\substack{(H_i, H_j)\in \cH\times \cH\\ H_i\cup H_j\cong S}} p^{e_S}\\
     &\leq 4^{v_H^2} \sum_{S\in\cS}  n^{v_S}p^{e_S}\\
     &\leBy{eq:exp_S-disjoint} 16^{v_H^2}\, n^{v_H}p^{e_H}\,,
\end{split}  
\end{equation}
where in the last step we bounded $|\cS|$ by the number of graphs on at most $2v_H$ vertices, which in turn is bounded by $\sum_{i=2}^{2v_H} 2^{\binom{i}{2}}\leq 2^{\binom{2v_H}{2}+1}\leq 2^{2v_H^2}.$

 Consider now the property
\begin{equation} \label{eq:E-disjoint}
  \cE:=\Big\{K\seq K_n\,\Big|\,  |\cH(K)|\geq \mu/2\Big\}\,.
\end{equation}
By Janson's inequality (Theorem~\ref{thm:janson}) we have 
 \begin{multline} \label{eq:PrE-disjoint}
\Pr(\Gnp\in\neg\cE)
  \leq\exp\(-\frac{\mu^2}{8(\mu+\Delta)}\) \leq \exp\(-\frac{1}{16}\cdot\min \Big\{\mu, \frac{\mu^2}{\Delta}\Big\}\)\\
  \leByM{\eqref{eq:mu-disjoint},\,\eqref{eq:Delta-disjoint}}\quad \exp\(-\frac{1}{16}
  \min\Big\{c_0, \frac{c_0^2}{16^{v_H^2} }\Big\}n^{v_H}p^{e_H}\)
 \leBy{eq:def_b-disjoint}\quad
 \exp(-bn^{v_H}p^{e_H})\,, 
\end{multline}
where in the second to last step we also used that $c_0\leq 1$ (see~\eqref{eq:def_rhoprime-disjoint}). 

For a given graph $K\seq K_n$, consider the auxiliary graph $\tG=\tG(K)$ on the vertex set
$V(\tG)=\cH(K)$, in which two vertices are connected by an edge if
and only if those two copies of $H$ are not edge-disjoint.  

Note that 
$$
 \EE[e(\tG(\Gnp))]=\Delta/2
$$
for $\Delta$ as in~\eqref{eq:Delta-disjoint} (the factor~$1/2$ is due to the
fact that the sum in~\eqref{eq:Delta-disjoint} is over \emph{ordered} pairs). Thus for the property
 \begin{equation} \label{eq:D-disjoint}
  \cD:=\Big\{K\seq K_n\,\Big|\,e(\tG(K))\leq \Delta\Big\}\,,
 \end{equation}
we obtain with Markov's inequality that
\begin{equation} \label{eq:PrD-disjoint}
  \Pr(\Gnp\in\cD)\geq 1/2\,.
\end{equation}

By definition of the auxiliary graph $\tG=\tG(K)$, 
any independent set in $\tG$ corresponds to a family $\tcH\seq
\cH(K)$ of pairwise edge-disjoint copies of $H$ in $F\cap K$.
Thus our definitions of $\cD$ and $\cE$ imply with Tur\'{a}n's Theorem (Theorem~\ref{thm:turan}) that any graph $K\in \cD\cap\cE$ contains a subfamily $\tcH\seq \cH(K)$ of pairwise
edge-disjoint copies of $H$ of size at least
\begin{multline} \label{eq:tcH-disjoint}
|\tcH| \geByM{\text{Thm.~\ref{thm:turan}}} \frac{v(\tG)^2}{v(\tG) + 2e(\tG)} \geq  \frac{1}{2}\cdot\min  \Big\{ v(\tG), \frac{v(\tG)^2}{2e(\tG)}\Big\}
 \\ \geByM{\eqref{eq:E-disjoint},\eqref{eq:D-disjoint}}  \frac{1}{16}\min\Big\{4\mu, \frac{\mu^2}{\Delta}\Big\}\geq b  n^{v_H}p^{e_H}\,,
\end{multline} 
where the last inequality follows analogously to \eqref{eq:PrE-disjoint}.  In other words, we have just shown that
\begin{equation*}
    \cD \cap  \cE \quad \seq \quad  \cA
\end{equation*}
or, equivalently,
\begin{equation} \label{eq:ahd-disjoint}
   \neg\cA \cap  \cD  \quad \seq \quad \neg\cE\,.
\end{equation}

Since $\neg\cA$ and $\cD$ are both decreasing, we obtain with the Harris inequality (Theorem~\ref{thm:fkg}) that
\begin{equation*}
\begin{split}
\Pr(\Gnp\in\neg\cA)&\leByM{\text{Thm.~\ref{thm:fkg}}} \Pr(\Gnp\in\neg\cA \,|\,\Gnp\in\cD)\\
&\leBy{eq:PrD-disjoint} 2 \Pr(\Gnp\in\neg \cA\cap \cD) \leBy{eq:ahd-disjoint} 2\Pr(\Gnp\in\neg\cE)\\
 &\leBy{eq:PrE-disjoint}  2\exp(-bn^{v_H}p^{e_H})\leq 2^{1-bn^{v_H}p^{e_H}} \,,
\end{split}
\end{equation*}
as claimed.
\end{proof}

\section{Proof of Theorem~\ref{thm:main result}}
\label{sec:proof}

As already mentioned, our proof of Theorem~\ref{thm:main result} proceeds by induction on $e(G)$, whereas~$H$ is considered fixed. In order for this induction to work, we will prove the following stronger statement. Recall that we introduced the set of $H$-edges $E_H(\Gnp)$ and the notion of $H$-covered copies in Definition~\ref{def:H-covered}. 

\begin{lemma}[Main lemma]
  \label{lemma:induction}
  Let $H$ be a nonempty graph with $v_H\geq 3$.
  For any nonempty graph $G$ satisfying $m_2(G)< x^*(H)$ there exist positive constants
  $a$, $b$, $C$, and $n_0$ such that for $n\geq n_0$ and
  \begin{equation}  \label{eq:p}
   Cn^{-1/m_2(G,H)} \leq p\leq
    n^{-1/m^*(H)}\,,
   \end{equation}
  with probability at least $1-2^{-bn^{v_H}p^{e_H}}$, every
  red-blue-coloring of $E_H(G_{n,p})$ that does not contain a blue copy of $H$
  contains at least $an^{v_G}(n^{v_H-2}p^{e_H})^{e_G}$ many $H$-covered red copies of $G$.
\end{lemma}

 Note that, because of~\eqref{eq:evGH}, the number of
$H$-covered red copies of $G$ guaranteed by Lemma~\ref{lemma:induction}
is of the same order of magnitude as the expected number of copies of
graphs from $\cF(G,H)$ (as defined in Definition~\ref{def:FGH}) in
$\Gnp$.

\begin{remark}
The statement of Lemma~\ref{lemma:induction} is void if $H$ is not strictly balanced \wrt $d_2(G,\cdot)$ (recall Lemma~\ref{lemma:equivalence}). On the other hand, for $H$ strictly balanced \wrt $d_2(G,\cdot)$, there is $p$ as in~\eqref{eq:p}, and -- crucially for our proof of Lemma~\ref{lemma:induction} -- we may apply Lemma~\ref{lemma:disjoint} for such~$p$. 
\end{remark}

\begin{remark}
For the two-color case studied here, it would be sufficient to prove the statement of Lemma~\ref{lemma:induction} with an error probability of $2^{-\Theta(n)}$ instead of $2^{-\Theta(n^{v_H}p^{e_H})}$ (see~\eqref{eq:rho-d-dense} below). However, our arguments yield the latter for free, and this is also what would be needed to extend our inductive approach to more than two colors.
\end{remark}

Lemma~\ref{lemma:induction} implies Theorem~\ref{thm:main result} as follows. 

\begin{proof}[Proof of Theorem~\ref{thm:main result}] Owing to
  Lemma~\ref{lemma:equivalence}, $G$ and $H$ as in Theorem~\ref{thm:main
    result} satisfy the hypothesis 
  of
  Lemma~\ref{lemma:induction}. We will prove Theorem~\ref{thm:main
    result} for the constant $C=C(G,H)$ guaranteed by
  Lemma~\ref{lemma:induction}.

By monotonicity it suffices to prove the theorem for $p=p(n):=Cn^{-1/m_2(G,H)}$. Again due to Lemma~\ref{lemma:equivalence}, this is smaller than $n^{-1/m^*(H)}$ for $n$ large enough, and thus Lemma~\ref{lemma:induction} is applicable for this $p=p(n)$. Clearly, if the event in Lemma~\ref{lemma:induction} holds then we have in particular that $\Gnp \rightarrow(G,H)$. Furthermore, due to~\eqref{eq:exponent-1} and the assumption that $G$ is not a matching, $n^{v_H}p^{e_H}$ is a growing function of $n$.  Hence the probability stated in Lemma~\ref{lemma:induction} is indeed $1-o(1)$, and Theorem~\ref{thm:main result} is proved.
\end{proof}

\subsection{Proof of Lemma~\ref{lemma:induction}}

It remains to prove Lemma~\ref{lemma:induction}, which we will do in the remainder of this section. Our main proof hinges on two fairly involved statements (Claim~\ref{clm:first-round} and Claim~\ref{clm:second-round} below), whose proofs are deferred to Section~\ref{sec:first-round} and Section~\ref{sec:second-round}, respectively.

As already mentioned, we proceed by induction on $e(G)$. Our induction base is the case where $G$ is a matching.

\begin{proof}[Proof of Lemma~\ref{lemma:induction}: Induction base --
  $G$ is a matching]
W.l.o.g.\ we may assume that $G$ contains no isolated vertices, i.e., that $v_G=2e_G$. 
Let
\begin{equation} \label{eq:def_rhoprime-basecase}
\rho:=\rho(H,1)\,, \qquad n':=n_0(H,1)\,, \qquad  b':=b(H,1)
\end{equation}
denote the constants obtained by applying Lemma~\ref{lemma:disjoint} for $H$ and $d:= 1$.
We shall prove Lemma~\ref{lemma:induction} for 
\begin{align}
  a=a(G,H)&:= \left(\frac{b'}{(2e_G)^{v_H}}\right)^{e_G}\,, \label{eq:def_a-base}\\ 
  b=b(G,H)&:=\frac{b'}{(\log_2 e_G + 2)\cdot  (2e_G)^{v_H}}\,,  \label{eq:def_b-base}\\
  C=C(G,H)&:=b^{-1/e_H} \label{eq:def_C-base}\,,\\
  n_0=n_0(G,H)&:=(2e_G) \cdot n'
  \,. \label{eq:def_n0-base} 
\end{align}
Note that for any $n\geq 1$ we have 
\begin{equation} \label{eq:b-large-enough-base}
1 = n^{2-1/m_2(G)}\leByM{\eqref{eq:exponent-1}, \eqref{eq:def_C-base}} b C^{e_H} n^{v_H - e_H/m_2(G,H)} \leBy{eq:p} bn^{v_H}p^{e_H}  \,.
\end{equation}

Fix $e_G$ pairwise disjoint sets $V_1, \dots, V_{e_G} \seq [n]$ of size \begin{equation} \label{eq:tn-base}
  \tn:=\lfloor n/e_G \rfloor\geq n/(2e_G)
\end{equation} each, and note that the graphs $\Gnp[V_i]$ induced by $\Gnp$ on these sets behave like independent random graphs $G_{\tn, p}$, where $\tn\geq n'$ due to our choice of $n_0$ in~\eqref{eq:def_n0-base}. 

Due to our choice of constants in~\eqref{eq:def_rhoprime-basecase} and observing that the complete graph $K_\tn$ is $(\rho, 1)$-dense,
we obtain with Lemma~\ref{lemma:disjoint} and the union bound that for $n\geq n_0$, with probability at least
\begin{multline*}
 1-e_G\cdot 2^{1-b'\tn^{v_H}p^{e_H}} \geq 1-2^{\log_2 e_G + 1 - (b'/(2e_G)^{v_H}) n^{v_H}p^{e_H}}\\
 \eqBy{eq:def_b-base}
 1-2^{\log_2 e_G + 1 - (\log_2 e_G + 2)bn^{v_H}p^{e_H}} 
 \geBy{eq:b-large-enough-base} 1-2^{-bn^{v_H}p^{e_H}}\,,
\end{multline*}
each of the graphs $\Gnp[V_i]$ contains a family of at least
$b'\tn^{v_H}p^{e_H}$ pairwise edge-disjoint copies of $H$. To avoid creating a blue copy of $H$, one edge from each of these copies needs to be colored red. Thus in every red-blue-coloring of $E_H(\Gnp)$ there is either a blue copy of~$H$ or we can obtain at least $$\(b'\tn^{v_H}p^{e_H}\)^{e_G}\geByM{\eqref{eq:def_a-base},\,\eqref{eq:tn-base}} an^{2e_G}  (n^{v_H-2}p^{e_H})^{e_G} =  an^{v_G}  (n^{v_H-2}p^{e_H})^{e_G}$$ many red matchings by picking exactly one $H$-covered red edge from each of the graphs $\Gnp[V_i]$, $1\leq i\leq e_G$. By definition, these red matchings are $H$-covered red copies of~$G$.
\end{proof}

Before giving the proof of the induction step, let us give an informal
outline of the key proof ideas. As already mentioned, our approach can be seen as a refinement of the proof for the symmetric case given by R\"odl and Ruci\'{n}ski in~\cite{MR1276825}. We will generate $\Gnp$ in two rounds,
i.e., as the union of two independent binomial random graphs
$G_{n,p_1}$ and $G_{n,p_2}$ on the same vertex set. Let $G_-$ denote a
fixed subgraph of $G$ with $e(G)-1$ edges and $v(G)$ vertices. By the
induction hypothesis, with high probability every coloring of the
$H$-edges of the first round that does not contain a blue copy of $H$
contains `many' $H$-covered red copies of $G_-$. Each of those induces
a vertex pair that will complete a red copy of $G$ if it is sampled as
an edge of the second round and is colored red. In our argument we will consider vertex pairs that complete not only one, but `many' red copies of $G_-$ to copies of~$G$. We will call the graph spanned by these edges the \emph{base graph} $\Gamma(h)$ of a given coloring~$h$ of $E_H(G_{n,p_1})$, the $H$-edges of the first round.  Our main goal when analyzing the first round is to show that, with suitably high probability, the base graph $\Gamma(h)$ is $(\rho,d)$-dense  for every coloring~$h$ of $E_H(G_{n,p_1})$ (for appropriately chosen parameters $\rho$ and $d$). Once this is shown, we may apply Lemma~\ref{lemma:disjoint} to find `many' pairwise edge-disjoint copies of $H$ in $\Gamma(h)\cap G_{n,p_2}$, the random subgraph of $\Gamma(h)$ spanned by the edges of the second round. In order to avoid creating a blue copy of $H$, one edge from each such copy needs to be colored red, which by definition of the base graph $\Gamma(h)$ creates `many' $H$-covered red copies of $G$.

For this approach to work, the arguments of the second round need to work for all possible colorings of the $H$-edges of the first round \emph{simultaneously}. In order to infer this with the union bound, we need that for a \emph{fixed} coloring $h$ of the first round, the second round fails with probability exponentially small in the number of $H$-edges. Here it is crucial that we only consider colorings of the $H$-edges of the first round, as the error probability for the second round is not small enough to beat the number of colorings of \emph{all} edges of the first round!

\begin{proof}[Proof of Lemma~\ref{lemma:induction}: Induction step -- $G$ is not a matching]
We denote by $G_-$ an arbitrary fixed subgraph of $G$ with $e(G)-1$ edges and $v(G)$ vertices. Note that we imposed no balancedness restricion on $G$, and hence both~$G$ and $G_-$ may be disconnected and even contain isolated vertices.

We start by fixing all constants needed in the proof. Throughout the following, by $a(G_-, H)$ etc we denote the constants guaranteed inductively by Lemma~\ref{lemma:induction}.

Let
\begin{equation}
  \label{eq:def_ell}
  \ell:=2\big(v_G+(e_G-1)(v_H-2)\big)\,,
\end{equation}
and set
\begin{equation} \label{eq:def_dprime}
d:= \frac{a(G_-,H)^2}{12\cdot 2^{\ell^2}\cdot\ell^{2v_G}}\,.
\end{equation}

Let
\begin{equation} \label{eq:def_rhoprime}
\rho:=\rho(H,d)\,, \qquad n':=n_0(H,d)\,, \qquad  b':=b(H,d)
\end{equation}
denote the constants obtained by applying Lemma~\ref{lemma:disjoint} for $H$ and $d$.
Set 
\begin{equation} \label{eq:def_cGamma}
c_\Gamma:=\rho^{(v_G-2)+(e_G-1)(v_H-2)} \cdot a(G_-, H)
\end{equation}
and
\begin{equation} \label{eq:def_C1}
C_1:=\max\left\{C(G_-, H), \(\frac{3}{c_\Gamma} \)^{\frac{1}{e_H(e_G-1)}}\right\}.
\end{equation}

Fix $\alpha>0$ small enough such that
\begin{equation} \label{eq:def_alpha}
 \alpha^{e_H}\leq\frac{b'}{16e_H} \quad \text{and} \quad 
 (1-\alpha)^{e_H}\geq 1/2\,,
\end{equation}
and set
\begin{align}
  b_1&:=\frac12\,b(G_-,H)\rho^{v_H} \alpha ^{e_H}
  \,,  \label{eq:def_bOne}\\ 
  b_2&:= b'/4\,. \label{eq:def_bTwo}
\end{align}

We shall prove Lemma~\ref{lemma:induction} for 
\begin{align}
  a=a(G,H)&:= (b'/2)\cdot c_{\Gamma} \cdot \alpha^{(e_G-1)e_H}\,, \label{eq:def_a}\\ 
  b=b(G,H)&:=\frac12\min\{b_1, b_2/2\}\,,  \label{eq:def_b}\\
  C=C(G,H)&:=\max\left\{\frac{C_1}{\alpha \cdot\rho^{1/m_2(G,H)}}, b^{-1/e_H}\right\} \label{eq:def_C}\,,\\
  n_0=n_0(G,H)&:=\max\left\{\frac{n_0(G_-, H)}{\rho}, n'\right\}
  \,. \label{eq:def_n0} 
\end{align}

Let $n\geq n_0$ and $p$ as in \eqref{eq:p}  be given, and set
\begin{equation} \label{eq:def_pi}
p_1:= \alpha p\,, \qquad p_2:= \frac{p-p_1}{1-p_1}\,.
\end{equation}
Note that
\begin{equation} \label{eq:p2}
(1-\alpha)p \leq  p_2\leq p \,.
\end{equation}

Throughout the proof we will identify $G_{n,p}$ with the union of two independent random graphs $G_{n,p_1}$ and $G_{n,p_2}$ on the same vertex set $[n]$. Note that indeed each edge
of $K_n$ is included in $G_{n,p_1}\cup G_{n,p_2}$ with probability
$$
1-(1-p_1)(1-p_2) \eqBy{eq:def_pi} p
$$
independently. 

As $G$ is not a matching, we have $m_2(G)\geq 1$, and consequently for any $n\geq 1$ that
\begin{equation} \label{eq:b-large-enough}
n \leq n^{2-1/m_2(G)}\leByM{\eqref{eq:exponent-1}, \eqref{eq:def_C}} b C^{e_H} n^{v_H - e_H/m_2(G,H)} \leBy{eq:p} bn^{v_H}p^{e_H}  \,.
\end{equation}

Next we define a number of graph properties to formalize the ideas outlined above. Throughout, $\cA, \cB, \cC$ etc.\ denote `good' properties, i.e., properties that are desirable in our proofs.

Let 
\begin{equation} \label{eq:A}
 \cA:=\left\{K\seq K_n \bigmid \parbox{0.55\displaywidth}{Every red-blue-coloring of $E_H(K)$ that does not contain a blue copy of $H$ contains at least $an^{v_G}(n^{v_H-2}p^{e_H})^{e_G}$ many $H$-covered red copies of
$G$}\right\} \,,
\end{equation}
 and note that $\cA$ is an increasing graph property. Our goal is to bound
$\Pr(\Gnp\in\neg\cA)$ from above.

For any graph $K\seq K_n$ (representing a fixed outcome of $G_{n,p_1}$) and any red-blue coloring~$h$ of $E_H(K)$, set
\begin{equation} \label{eq:AKh}
 \cA_{K,h}:=\left\{K'\seq K_n \bigmid \parbox{0.55\displaywidth}{Every extension of $h$ to $E_H(K \cup K')$ that does not contain a blue copy of $H$ contains at least $an^{v_G}(n^{v_H-2}p^{e_H})^{e_G}$ many $H$-covered red copies of $G$} \right\} \,.
\end{equation} 
Note that $\cA_{K,h}$ is increasing for any fixed $K$ and $h$.

 Let
\begin{equation} \label{eq:def_z}
 z:=c_\Gamma \cdot n^{v_G-2}\(n^{v_H-2}p_1^{e_H}\)^{e_G-1}\,.
\end{equation}
where $c_\Gamma$ is defined in~\eqref{eq:def_cGamma}. For any graph $K\seq K_n$ (again representing a fixed outcome of $G_{n,p_1}$) and any red-blue-coloring $h$ of $E_H(K)$, set
\begin{equation} \label{eq:def_base}
 \Gamma(K,h):=\left\{\textstyle e\in \binom{[n]}{2}\bigmid \parbox{0.55\displaywidth}{$e$
    completes at least $z$ many $H$-covered copies of
    $G_-$ in 
    $E_H(K)$ that are colored red in $h$ to copies of
    $G$}\right\} \,.
\end{equation}
We will refer to the graph~$([n],\Gamma(K,h))\seq K_n$ as the \textit{base graph}
determined by the coloring~$h$. Further, let 
\begin{equation} \label{eq:B}
 \cB:=\left\{K\seq K_n \bigmid \parbox{0.55\displaywidth}{For every red-blue-coloring $h$ of $E_H(K)$ that does not contain a blue copy of $H$, the base graph $\Gamma(K,h)$ is $(\rho, d)$-dense}\right\} \,,
\end{equation}
where $d$ and $\rho$ are defined in~\eqref{eq:def_dprime} and ~\eqref{eq:def_rhoprime}. Note that $\cB$ is an increasing graph property. Finally, let \begin{equation} \label{eq:C}
\cC:=\Big\{K\seq K_n\,\Big|\, |E_H(K)|\leq 2e_H \cdot n^{v_H}p_1^{e_H} \Big\}\,,
\end{equation}
and note that $\cC$ is a decreasing graph property.

We will prove the following two claims.

\begin{claim} \label{clm:first-round} We have
$$\Pr(\GnpOne \in \neg\cB)\leq 2^{-b_1\, n^{v_H}p^{e_H}}\,.$$
\end{claim}

\begin{claim} \label{clm:second-round}
For every $K\in \cB$ and every red-blue-coloring $h$ of $E_H(K)$, we have
$$\Pr(\GnpTwo\in\neg\cA_{K,h}) \leq 2^{-b_2\,n^{v_H}p^{e_H}}\,.$$
\end{claim}

Claim~\ref{clm:first-round} and Claim~\ref{clm:second-round} imply Lemma~\ref{lemma:induction} as follows. Recall that our goal is to bound
$\Pr(\Gnp\in\neg\cA)$ from above, and that we generate $\Gnp$ as the union of two independent random graphs $\GnpOne$ and $\GnpTwo$.

As the expected number of copies of $H$ in $G_{n,p_1}$ is bounded by $n^{v_H}p_1^{e_H}$, Markov's inequality yields for $\cC$ defined in \eqref{eq:C} that
\begin{equation} \label{eq:prc}
 \Pr(\GnpOne\in\cC)\geq 1/2\,.
\end{equation}

For any graph $K'\seq K_n$ (representing a fixed outcome of $G_{n,p_2}$) we set $$\cA_{K'}:=\Big\{K\seq K_n \,\Big|\,K\cup K'\in \cA\Big\}\,,$$
where $\cA$ is defined in~\eqref{eq:A}. As $\cA$ is increasing, also the property $\cA_{K'}$ is increasing for any $K'\seq K_n$. Thus its complement is decreasing, and we obtain with the Harris inequality (Theorem~\ref{thm:fkg}) that for any $K'\seq K_n$ we have 
\begin{equation} \label{eq:FKG-conditional}
\Pr(\GnpOne\in\neg\cA_{K'})\leByM{\text{Thm.~\ref{thm:fkg}}} \Pr(\GnpOne\in\neg\cA_{K'} \,|\,\GnpOne\in\cC)\leBy{eq:prc} 2 \Pr(\GnpOne\in\neg\cA_{K'}\cap \cC)\,.
\end{equation}
Using the independence of $\GnpOne$ and $\GnpTwo$ and the law of total probability, we can infer that
\begin{equation} \label{eq:FKG}
\begin{split}
\Pr(\Gnp\in\neg\cA)& =\sum_{K'\seq K_n}\Pr[\GnpOne\in\neg\cA_{K'}]\cdot\Pr[\GnpTwo=K']
\\&\leBy{eq:FKG-conditional} 2 \(\sum_{K'\seq K_n}\Pr[\GnpOne\in\neg\cA_{K'}\cap \cC]\cdot\Pr[\GnpTwo=K']\)\\
&= 2\, \Pr(\(\Gnp\in\neg\cA\) \,\wedge\, \(\GnpOne\in\cC\))\,.
\end{split}
\end{equation}

Thus it suffices to bound the last probability. Again by the law of total probability, we have
\begin{equation} \label{eq:split-up}
\begin{split}
 &\Pr\big((\Gnp\in\neg\cA)\,\wedge\, (\GnpOne\in\cC)\big)\\
 =\, &\Pr\big((\Gnp\in\neg\cA)\, \wedge\, (\GnpOne\in\neg \cB\cap \cC)\big)\\
 &\quad+ \sum_{K\in \cB\cap \cC} \Pr\big(\Gnp\in\neg\cA\,|\,G_{n,p_1}=K\big)\Pr(G_{n,p_1}=K)\\
 \leq\, &\Pr(\GnpOne\in\neg \cB) + \max_{K\in \cB\cap \cC} \Pr(\Gnp\in\neg\cA\,|\,G_{n,p_1}=K)\,.
 \end{split}
\end{equation}

Observe that if $\GnpOne=K$ we have $\Gnp\in\neg\cA$ if and only if $\GnpTwo\in\;\bigcup_{h}\;\neg\cA_{K,h}$, where the union is over all $2^{|E_H(K)|}$ red-blue colorings $h$ of $E_H(K)$. Together with the independence of $\GnpOne$ and $\GnpTwo$ it follows that for any $K\seq K_n$ we have
\begin{equation*}
\begin{split}
\Pr(\Gnp\in\neg \cA\,|\,G_{n,p_1}=K)
& =  \Pr(\GnpTwo\in\textstyle\bigcup_{h}\neg\cA_{K,h})\,.
\end{split}
\end{equation*}
If in addition $K$ is in~$\cC$ as defined in~\eqref{eq:C}, we obtain with $$|E_H(K)|\leBy{eq:C}2e_H \cdot n^{v_H}p_1^{e_H}  \leByM{\eqref{eq:def_alpha}, \eqref{eq:def_pi}}(b'/8) n^{v_H}p^{e_H} \eqBy{eq:def_bTwo}   (b_2/2)  n^{v_H}p^{e_H}
$$ and the union bound that
\begin{equation} \label{eq:KinC}
\Pr(\Gnp\in\neg \cA\,|\,G_{n,p_1}=K) \leq
2^{(b_2/2) n^{v_H}p^{e_H}}\cdot \max_{h}\,
\Pr(\GnpTwo\in\neg\cA_{K,h})\,,
\end{equation}
where the maximum is over all red-blue colorings $h$ of $E_H(K)$.

Combining~\eqref{eq:FKG}, \eqref{eq:split-up},  and \eqref{eq:KinC}, we obtain that
\begin{multline*} 
\prob{\Gnp\in\neg \cA}
\leq 2 \( \Pr(\GnpOne\in\neg\cB) + 2^{(b_2/2) n^{v_H}p^{e_H}} \max_{K\in \cB,\, h} \Pr(\GnpTwo\in\neg\cA_{K,h})\)\\
\leByM{\text{Cl.~\ref{clm:first-round},\,Cl.~\ref{clm:second-round}}} 2\(2^{-b_1n^{v_H}p^{e_H}} +   2^{-(b_2/2)n^{v_H}p^{e_H}} \)\\
\leq 4\cdot2^{-\min\{b_1, b_2/2\} n^{v_H}p^{e_H} } \eqBy{eq:def_b}2^{2-2b n^{v_H}p^{e_H} } \leq 2^{-bn^{v_H}p^{e_H}}\,,
\end{multline*}
where in the last step we used that $2\leq bn^{v_H}p^{e_H}$ due to \eqref{eq:b-large-enough}.
\end{proof}

It remains to prove Claim~\ref{clm:first-round} and Claim~\ref{clm:second-round}.

\subsection{Proof of Claim~\ref{clm:first-round}}
\label{sec:first-round}
We start with the proof of Claim~\ref{clm:first-round}, which concerns
the `probability of failure' of the first round~$G_{n,p_1}$.

\begin{proof}[Proof of Claim~\ref{clm:first-round}]
  In order to verify that a graph $F\seq K_n$ is $(\rho,d)$-dense, an averaging argument
  shows that it suffices to check that every set $V\seq [n]$ of size
  \begin{equation} \label{eq:def_tn}
    \tn:=\lceil\rho n \rceil 
  \end{equation} 
  contains at least $d\binom{\tn}{2}$ edges of $F$ (see~\cite{MR1276825}). 

For any  graph $K\seq K_n$ (representing a fixed outcome of $\GnpOne$), any red-blue-coloring $h$ of $E_H(K)$, and any set $V\seq [n]$, $|V|=\tn$,  set
\begin{equation} \label{eq:def_base_V} 
 \Gamma(K,h,V):=\left\{\textstyle e\in \binom{V}{2} \bigmid \parbox{0.55\displaywidth}{$e$
    completes at least $z$ many $H$-covered copies of
    $G_-$ in 
    $E_H(K[V])$ that are colored red in $h$ to copies of
    $G$}\right\} \,,
\end{equation}
where $z$ is defined in~\eqref{eq:def_z}, and define
\begin{equation} \label{eq:BV}
 \cB_V:=\left\{K\seq K_n \bigmid \parbox{0.55\displaywidth}{For every red-blue-coloring $h$ of $E_H(K)$ that does not contain a blue copy of $H$, we have $|\Gamma(K,h,V)|\geq d\tn^2$}\right\} \,.
\end{equation}
 Note that $\cB_V$ is increasing.

For a fixed set $V\seq [n]$, $|V|=\tn$, and for any red-blue coloring $h$ of $E_H(K)$, let $k_{G_-}(K,h,V)$ denote the total number of $H$-covered red copies of $G_-$ in $E_H(K[V])$, and set
\begin{equation} \label{eq:AV}
 \cA_V:=\left\{K\seq K_n \bigmid \parbox{0.6\displaywidth}{For every red-blue-coloring $h$ of $E_H(K)$ that does not contain a blue copy of $H$, we have\\ $k_{G_-}(K,h,V)\geq a(G_-,H)\cdot \tn^{v_G}(\tn^{v_H-2}p_1^{e_H})^{e_G-1}$}\right\} \,.
\end{equation}
Note that $\cA_V$ is increasing.

Recall that $H$-covered copies of $G_-$ are copies of $G_-$ that are a central copy in a copy of a graph $G_-^H\in\cF(G_-,H)$ as defined in Definition~\ref{def:FGH}. Let $\cT$ be the family of all pairwise nonisomorphic graphs $T$ which
are unions of two graphs from $\cF(G_-,H)$, say $G^H_{1-}$ and~$G^{H}_{2-}$, such that some vertex pair $g\in \binom{V(T)}{2}$
completes both a central copy in $G^H_{1-}$ and a central copy in
$G^H_{2-}$ to a copy of $G$. For any graph $K\seq K_n$ and any set $V\seq [n]$, $|V|=\tn$, let $k_\cT(K,V)$ denote the number of
copies of graphs from $\cT$ in $K[V]$, and set 
\begin{equation} \label{eq:def_DV}
 \cD_V:=\Big\{K\seq K_n \,\Big|\, k_\cT(K,V) \leq 2^{\ell^2} \tn^{2v_G-2} (\tn^{v_H-2}p_1^{e_H})^{2(e_G-1)}\Big\} \,.
\end{equation}
Note that $\cD_V$ is decreasing.

We will show the following three statements.

\begin{fact} \label{fact:PrAV} For every fixed set $V\seq [n]$, $|V|=\tn$, we have
 $\Pr(\GnpOne\in\neg\cA_V)\leq 2^{-2b_1\, n^{v_H}p^{e_H}}$.
\end{fact}

\begin{fact} \label{fact:PrDV} For every fixed set $V\seq [n]$, $|V|=\tn$, we have
 $\Pr(\GnpOne\in\cD_V)\geq 1/2$.
\end{fact}

\begin{fact} \label{fact:jensen} For every fixed set $V\seq [n]$, $|V|=\tn$, we have
 $\cA_V \cap  \cD_V \, \seq \, \cB_V$.
\end{fact}

With these statements in hand, Claim~\ref{clm:first-round} can be deduced as follows.
Note that Fact~\ref{fact:jensen} is equivalent to
\begin{equation*}    \neg\cB_V \cap  \cD_V  \, \seq \, \neg\cA_V\,.
\end{equation*}
Since $\neg\cB_V$ and $\cD_V$ are both decreasing, we obtain
with the Harris inequality (Theorem~\ref{thm:fkg}) that

\begin{equation} \label{eq:PrBVC}
\begin{split}
\Pr(\GnpOne\in\neg\cB_V)&\leByM{\text{Thm.~\ref{thm:fkg}}} \Pr(\GnpOne\in\neg\cB_V  \,|\,\GnpOne\in\cD_V)\\
&\leByM{\text{Fact~\ref{fact:PrDV}}} 2\, \Pr(\GnpOne\in\neg \cB_V\cap \cD_V)\\
&\leByM{\text{Fact~\ref{fact:jensen}}} 2 \Pr(\GnpOne\in\neg\cA_V)\\
&\leByM{\text{Fact~\ref{fact:PrAV}}} 2^{1-2b_1\,n^{v_H}p^{e_H}}\,.\\
\end{split}
\end{equation}
By definition of $\cB$ and $\cB_V$ (see~\eqref{eq:B} and~\eqref{eq:BV}), we have
\begin{equation} \label{eq:nbc}
   \neg\cB  \quad \seq \quad \bigcup_{\substack{V\seq [n]: \\
    |V|=\tn}} \neg \cB_V \,.
\end{equation}
Taking the union bound over all sets $V\seq [n]$ with $|V|=\tn$ we obtain
\begin{equation} \label{eq:rho-d-dense}
\Pr(\GnpOne\in\neg \cB)\leBy{eq:nbc} \sum_{\substack{V\seq [n]: \\
    |V|=\tn}}\Pr(\GnpOne\in\neg \cB_V) \leBy{eq:PrBVC}
2^{n+1-2b_1\,n^{v_H}p^{e_H}} \leq
2^{-b_1n^{v_H}p^{e_H}}\,,
\end{equation}
where in the last step we used that $n+1\leq b_1n^{v_H}p^{e_H}$ due to~\eqref{eq:b-large-enough} and~\eqref{eq:def_b}.
\end{proof}

It remains to prove Facts~\ref{fact:PrAV}, \ref{fact:PrDV}, and~\ref{fact:jensen}. For all these proofs, note that $G_{n,p_1}[V]$ behaves exactly like a binomial random graph $G_{\tn,p_1}$, and that
\begin{equation}
  \label{eq:p1-large-enough}
  \qquad p_1 \geByM{\eqref{eq:p},\,\eqref{eq:def_pi}} \alpha C
  n^{-1/m_2(G,H)} 
  \geByM{\eqref{eq:def_C},\,\eqref{eq:def_tn}}  C_1\, \tn^{-1/m_2(G,H)}\,.
\end{equation}

\begin{proof}[Proof of Fact~\ref{fact:PrAV}] 
Owing to the monotonicity of the $2$-density and to the assumption on $G$ in Lemma~\ref{lemma:induction}, we have
$m_2(G_-)\leq m_2(G)< x^*(H)$. Moreover, by our choice of
constants, we have
$$\tn \geBy{eq:def_tn} \rho n_0 \geBy{eq:def_n0} n_0(G_-,H)$$
and
\begin{equation*}  C(G_-,H) \widetilde{n}^{-1/m_2(G_-,H)} \leByM{\eqref{eq:def_C1}, \eqref{eq:p1-large-enough}} p_1 \leByM{\eqref{eq:p},\,\eqref{eq:def_pi}} \alpha n^{-1/m^*(H)}
  \leq \tn^{-1/m^*(H)}\,.
\end{equation*}
Thus we may apply the induction hypothesis to $G_{n,p_1}[V]$ to infer
\begin{equation*} 
  \Pr(\GnpOne\in\neg\cA_V)\leq 2^{-b(G_-,H)\,\tn^{v_H}p_1^{e_H}}\leByM{\eqref{eq:def_bOne}, \eqref{eq:def_pi}, \eqref{eq:def_tn}}2^{-2b_1\, n^{v_H}p^{e_H}}
\end{equation*}
recalling the definition of $\cA_V$ in~\eqref{eq:AV}.
\end{proof}

\begin{proof}[Proof of Fact~\ref{fact:PrDV}] Consider a fixed graph
  $T\in\cT$ as defined before~\eqref{eq:def_DV}, and let
  \[
  	J:=G^H_{1-}\cap G^H_{2-}
  \] 
  denote the intersection of the two graphs
  from $\cF(G_-,H)$ forming~$T$. We obtain with
  Lemma~\ref{lemma:goalJ} that
  \begin{equation} \label{eq:exp-J} \tn^{v(J)}
    p_1^{e(J)}\geBy{eq:p1-large-enough}\tn^{v(J)-e(J)/m_2(G,H)}\geByM{\text{L.10}}
    \tn^{2 - L(G_{1-}^H)}\,,
  \end{equation}
where in the first step we also used that $C_1\geq 1$. Thus the
expected number of copies of~$T$ in $G_{n,p_1}[V]$ is bounded by
\begin{equation*}
 \begin{split}
  \tn^{v(T)}p_1^{e(T)}   &\leBy{eq:exp-J} \tn^{v(G^H_{1-}) + v(G^H_{2-})-2+L(G^H_{1-})}
    p_1^{e(G^H_{1-}) + e(G^H_{2-})}\\ 
      &\eqBy{eq:evGH}   \tn^{2\big(v_G +
      (e_G-1)(v_H-2)\big)-2 - L(G^H_{2-})}  p_1^{2(e_G-1) e_H}\\
 &\leq\tn^{2v_G-2} (\tn^{v_H-2}p_1^{e_H})^{2(e_G-1)}     
       \,, 
\end{split}
\end{equation*}
where in the last step we used that $L(G^H_{2-})$ is nonnegative. Thus in total the expected number of graphs from $\cT$ in $G_{n,p_1}[V]$ is at most
\begin{equation*} 
2^{\ell^2-1} \cdot \tn^{2v_G-2} (\tn^{v_H-2}p_1^{e_H})^{2(e_G-1)}
\end{equation*}
where we bounded $|\cT|$ by the number of graphs on at most
$\ell$ vertices (recall~\eqref{eq:def_ell}), which in turn
is bounded by $\ell 2^{\binom{\ell}{2}}\leq 2^{\ell^2-1}$.

Fact~\ref{fact:PrDV} now follows with Markov's inequality, recalling the definition of $\cD_V$ in~\eqref{eq:def_DV}.
\end{proof}

\begin{proof}[Proof of Fact~\ref{fact:jensen}] Consider a fixed set
  $V\seq[n]$, $|V|=\tn$, and an arbitrary graph $K\seq K_n$. For any
  red-blue-coloring $h$ of $E_H(K)$ and for every edge $e\in
  \binom{V}{2}$, let
\begin{equation*} 
 x_e=x_e(K,h,V):=\left|\left\{\overline{G_-} \bigmid \parbox{0.55\displaywidth}{$\overline{G_-}$ is an $H$-covered copy of $G_-$ in $E_H(K[V])$ that is colored red in $h$, and $e$ completes $\overline{G_-}$ to a copy of $G$}\right\}\right| \,.
\end{equation*}

Note that, by our definition of $\Gamma(K,h,V)$ in~\eqref{eq:def_base_V}), for all $e\in \Gamma(K,h,V)$ we have
\begin{equation}
  \label{eq:xe-large}
  x_e  \geq z \geByM{\eqref{eq:def_z}, \eqref{eq:p1-large-enough}}  c_\Gamma\, C_1^{e_H(e_G-1)}\,n^{v_G-2}
  (n^{v_H-2}\tn^{-e_H/m_2(G,H)})^{e_G-1}
  \geByM{\eqref{eq:exponent-2},\eqref{eq:def_C1}} 3\,,
\end{equation}
where in the second inequality we also used that $\tn\leq n$.

We will show that if $K$ is in $\cA_V$, we have
\begin{equation}
  \label{eq:ifAVholds}
\sum_{e\in \Gamma(K,h,V)}x_e\geq \frac{a(G_-,H)}{2} \cdot  \tn^{v_G}(\tn^{v_H-2}p_1^{e_H})^{e_G-1}
\end{equation}
for every coloring $h$ of $E_H(K)$ that contains no blue copy
of~$H$, and that if $K$ is in $\cD_V$, we have
\begin{equation} \label{eq:ifDVholds}
\sum_{e\in \Gamma(K,h,v)}\binom{x_e}{2}\leq
  2^{\ell^2}\ell^{2v_G}\cdot \tn^{2v_G-2}
   (\tn^{v_H-2}p_1^{e_H})^{2(e_G-1)} 
\end{equation}
for every coloring $h$ of $E_H(K)$.

As by Jensen's inequality we have
\begin{equation} 
  \label{eq:jensen}
  \sum_{e\in \Gamma(K,h,V)}\binom{x_e}{2} \geq
  \big|\Gamma(K,h,V)\big|
  \binom{\big|\Gamma(K,h,V)\big|^{-1}\sum_{e\in\Gamma(K,h,V)}x_e}{2}
  \geBy{eq:xe-large}
  \frac  {(\sum_{e\in
      \Gamma(K,h,V)}x_e)^2}{3\big|\Gamma(K,h,V)\big|}\,, 
\end{equation}
(where in the second inequality we used that $\binom{x}{2}\geq x^2/3$ for $x\geq 3$), \eqref{eq:ifAVholds} and \eqref{eq:ifDVholds} will imply that for any $K\in\cA_V \cap \cD_V$ we have
\begin{equation*} 
\big|\Gamma(K,h,V)\big|
\geBy{eq:jensen} \frac{(\sum_{e\in \Gamma(K,h,V)}x_e)^2}{3\sum_{e\in \Gamma(K,h,V)} \binom{x_e}{2}} \geByM{\eqref{eq:ifAVholds},\eqref{eq:ifDVholds}} \frac{a(G_-,H)^2}{12\cdot 2^{\ell^2}\cdot\ell^{2v_G}}\cdot \tn^2 \eqBy{eq:def_dprime} d\tn^2
\end{equation*}
for every coloring $h$ of $E_H(K)$ that does not contain a blue copy of $H$, i.e., that $K$ satisfies~$\cB_V$ as defined in~\eqref{eq:BV}.

It remains to show \eqref{eq:ifAVholds} and \eqref{eq:ifDVholds}. To verify \eqref{eq:ifAVholds}, recall that $k_{G_-}(K,h,V)$ denotes the total number of $H$-covered red copies of $G_-$ in $E_H(K[V])$. Since every such copy contributes to at least one of the $x_e$, we have
\begin{equation} \label{eq:sum_xe-0}
  \sum_{e\in \binom{V}{2}} x_e \geq k_{G_-}(K,h,V)\,.
\end{equation}
Note that
\begin{equation}  \label{eq:z-explained}
z\eqBy{eq:def_z}c_\Gamma \cdot n^{v_G-2} (n^{v_H-2}p_1^{e_H})^{e_G-1}\\
\leByM{\eqref{eq:def_cGamma},\,\eqref{eq:def_tn}} a(G_-,H)\cdot
\tn^{v_G-2} (\tn^{v_H-2}p_1^{e_H})^{e_G-1}\,. 
\end{equation}
Since by definition of $\Gamma(K,h,V)$ we have $x_e<z$ for all $e\in
\binom{V}{2}\setminus \Gamma(K,h,V)$ (recall~\eqref{eq:def_base_V}), it
follows that
\begin{multline*} 
 \sum_{e\in \Gamma(K,h,V)}x_e\;
 \geByM{\eqref{eq:sum_xe-0}}\; k_{G_-}(K,h,V)-\sum_{e\in \binom{V}{2}
   \setminus\Gamma(K,h,V)}x_e
  \geq\; k_{G_-}(K,h,V)-\binom{\tn}{2}\cdot z\\
\geByM{\eqref{eq:z-explained}}\; k_{G_-}(K,h,V)- \frac{a(G_-,H)}{2}
\cdot\tn^{v_G}(\tn^{v_H-2}p_1^{e_H})^{e_G-1}\,. 
\end{multline*}
It follows from the definition of $\cA_V$ in~\eqref{eq:AV} that indeed \eqref{eq:ifAVholds} holds
for every coloring $h$ of~$E_H(K)$ that contains no blue copy
of~$H$ if $K$ is in $\cA_V$.

To verify~\eqref{eq:ifDVholds}, recall that every $H$-covered copy of $G_-$ is contained in a copy of a graph $G_-^H\in\cF(G_-,H)$ (see Definition~\ref{def:H-covered} and Definition~\ref{def:FGH}). It follows with
the definition of $k_\cT(K,V)$ (see the paragraph before \eqref{eq:def_DV}) that
\begin{equation}
  \label{eq:sum_xe-squared}
   \sum_{e\in \binom{V}{2}} \binom{x_e}{2}  \leq
  \ell^{2v_G}\, k_\cT(K,V)\,, 
\end{equation}
where $\ell$ is as defined in~\eqref{eq:def_ell}.  
Here the constant~$\ell^{2v_G}$ follows from the fact that 
a given copy of some $T\in\cT$ contributes at most
$((v_T)_{v_G})^2 \leq (v_T)^{2v_G}\leq \ell^{2v_G}$
to the sum. 

Consequently, if $K\in\cD_V$, the definition of $\cD_V$ in~\eqref{eq:def_DV} implies that
\begin{equation*} 
\sum_{e\in \Gamma(K,h,v)}\binom{x_e}{2}\leq
\sum_{e\in \binom{V}{2}}\binom{x_e}{2}
   \leByM{\eqref{eq:def_DV},\,\eqref{eq:sum_xe-squared}}
   2^{\ell^2}\ell^{2v_G} \cdot \tn^{2v_G-2}
   (\tn^{v_H-2}p_1^{e_H})^{2(e_G-1)} 
\end{equation*}
for every coloring $h$ of $E_H(K)$, as claimed in~\eqref{eq:ifDVholds}.
\end{proof}

\subsection{Proof of Claim~\ref{clm:second-round}}
\label{sec:second-round}
It now remains to prove Claim~\ref{clm:second-round}. 
 
\begin{proof}[Proof of Claim~\ref{clm:second-round}]
 Consider a fixed graph $K\in\cB$ and a fixed red-blue-coloring  $h$ of~$E_H(K)$. By definition of the event $\cB$ (recall \eqref{eq:B}), the graph $\Gamma(K,h)$ is $(\rho, d)$-dense.
  
  Note that due to \eqref{eq:def_alpha} and \eqref{eq:p2} we have 
  \begin{equation} \label{eq:b'}
  b'n^{v_H}p_2^{e_H} \geq (b'/2)n^{v_H}p^{e_H}\,.
  \end{equation}

Thus Lemma~\ref{lemma:disjoint} yields with \eqref{eq:p} and our choice of constants in~\eqref{eq:def_rhoprime} and \eqref{eq:def_n0} that with probability at least \begin{equation*}
\begin{split}
 1-2^{1-b'n^{v_H}p_2^{e_H}} \geByM{ \eqref{eq:b'},  \eqref{eq:def_bTwo}} 
 1-2^{1 - 2b_2 n^{v_H}p^{e_H}} 
 \geq 1-2^{-b_2n^{v_H}p^{e_H}}
\end{split}
\end{equation*}
(where in the last step we used that $1\leq b^{v_H}p^{e_H}\leq b_2 n^{v_H}p^{e_H}$ due to~\eqref{eq:b-large-enough} and~\eqref{eq:def_b}),
the graph $\Gamma(K,h)\cap G_{n,p_2}$ contains a family of at least
$b'n^{v_H}p_2^{e_H}$ pairwise edge-disjoint copies of $H$. 

To avoid creating a blue copy of $H$, one edge from each of these copies needs to be colored red, and by the definition of $\Gamma(K,h)$
(see~\eqref{eq:def_base}), each such edge that is colored red creates at least $z$ many $H$-covered red copies of $G$. Thus any extension of $h$ to a coloring of $E_H(K)\cup E_H(G_{n,p_2})\seq E_H(K\cup G_{n,p_2})$ creates a
blue copy of $H$ or at least 
\begin{equation*}
\begin{split}
b'n^{v_H}p_2^{e_H}\cdot z
& \geByM{\eqref{eq:def_z}, \eqref{eq:b'}} b'/2\cdot c_{\Gamma} \cdot n^{v_H}p^{e_H} \cdot n^{v_G-2}(n^{v_H-2}p_1^{e_H})^{e_G-1}\\
&\eqByM{\eqref{eq:def_a},\,\eqref{eq:def_pi}} a n^{v_G}(n^{v_H-2}p^{e_H})^{e_G}
\end{split}
\end{equation*}
many $H$-covered red copies of $G$. Thus $\GnpTwo$ is indeed in $\cA_{K,h}$ as defined in~\eqref{eq:AKh} with the claimed probability.
\end{proof}

\section{Concluding remarks}
\label{sec:concluding}

We believe that the proof for the two-color case given
here can be extended to the setting with more than two colors along
the lines of~\cite{MR1276825}. Namely, for given graphs $H_1, \dots,
H_k$ with $m_2(H_k)\leq \dots \leq m_2(H_2) < m_2(H_1)$ and $H_1$
strictly balanced \wrt $d_2(H_2, \cdot)$, one should be able to prove
that $\Pr(\Gnp\rightarrow (H_1, \dots, H_k))=1-o(1)$ if $p\geq
Cn^{-1/m_2(H_{2},H_1)}$ as follows: Clearly, in order to prove
$\Gnp\rightarrow (H_1, \dots, H_k)$ it suffices to prove
$\Gnp\rightarrow (G, \dots, G, H)$, where $G$ denotes the disjoint
union of $H_2, \dots, H_{k}$, and $H:=H_1$. Furthermore, it is not
hard to see that $m_2(G)=m_2(H_{2})$, and consequently also
$m_2(G,H)=m_2(H_{2},H_1)$. Thus it suffices to show that
$\Pr(\Gnp\rightarrow (G, \dots, G, H))=1-o(1)$ if $p\geq
Cn^{-1/m_2(G,H)}$. We believe that this can be done by combining the
approach via double induction (on $e(G)$ and the number of colors $k$)
used in~\cite{MR1276825} with the ideas presented in this paper. Note
that this implies using Lemma 1 of~\cite{MR1276825}, which relies on
the regularity lemma for dense graphs.

We do not pursue this further here. In our view, a more interesting next step
would be to extend the approach taken in~\cite{up:frs} to the
asymmetric scenario, with the goal of deriving $1$-statements for more
general settings, in particular for the hypergraph setting. This might also help in getting rid of the balancedness assumption on $H$ in the existing proofs.

 An altogether
different open question is the proof of the $0$-statement in
Conjecture~\ref{conj:asymmetric colorings}. With some extra work the
approach in~\cite{MR2531778} can be pushed through to prove the
$0$-statement for certain graphs $G$ and $H$ that are not complete,
but a general proof does not seem to be within reach of the known
methods.


\begin{bibdiv}
\begin{biblist}

\bib{MR0384579}{book}{
      author={Berge, C.},
       title={Graphs and hypergraphs},
     edition={revised},
   publisher={North-Holland Publishing Co.},
     address={Amsterdam},
        date={1976},
        note={Translated from the French by Edward Minieka, North-Holland
  Mathematical Library, Vol. 6},
      review={\MR{MR0384579 (52 \#5453)}},
}

\bib{MR0419285}{article}{
      author={Burr, S.~A.},
      author={Erd{\H{o}}s, P.},
      author={Lov\'asz, L.},
       title={On graphs of {R}amsey type},
        date={1976},
        ISSN={0381-7032},
     journal={Ars Combinatoria},
      volume={1},
      number={1},
       pages={167--190},
      review={\MR{0419285 (54 \#7308)}},
}

\bib{MR1601954}{book}{
      author={Chung, F.},
      author={Graham, R.},
       title={Erd{\H{o}}s on graphs},
   publisher={A K Peters Ltd.},
     address={Wellesley, MA},
        date={1998},
        ISBN={1-56881-079-2; 1-56881-111-X},
        note={His legacy of unsolved problems},
      review={\MR{MR1601954 (99b:05031)}},
}

\bib{up:cg}{unpublished}{
      author={Conlon, D.},
      author={Gowers, W.~T.},
       title={Combinatorial theorems in sparse random sets},
        note={Submitted},
}

\bib{MR0309498}{article}{
      author={Fortuin, C.~M.},
      author={Kasteleyn, P.~W.},
      author={Ginibre, J.},
       title={Correlation inequalities on some partially ordered sets},
        date={1971},
        ISSN={0010-3616},
     journal={Comm. Math. Phys.},
      volume={22},
       pages={89--103},
      review={\MR{MR0309498 (46 \#8607)}},
}

\bib{MR932121}{article}{
      author={Frankl, P.},
      author={R{\"o}dl, V.},
       title={Large triangle-free subgraphs in graphs without {$K_4$}},
        date={1986},
        ISSN={0911-0119},
     journal={Graphs Combin.},
      volume={2},
      number={2},
       pages={135--144},
         url={http://dx.doi.org/10.1007/BF01788087},
      review={\MR{MR932121 (89b:05124)}},
}

\bib{up:frs}{article}{
      author={Friedgut, E.},
      author={R{\"o}dl, V.},
      author={Schacht, M.},
       title={Ramsey properties of random discrete structures},
        date={2010},
        ISSN={1042-9832},
     journal={Random Structures Algorithms},
      volume={37},
      number={4},
       pages={407--436},
         url={http://dx.doi.org/10.1002/rsa.20352},
      review={\MR{2760356}},
}

\bib{MR2187740}{incollection}{
      author={Gerke, S.},
      author={Steger, A.},
       title={The sparse regularity lemma and its applications},
        date={2005},
   booktitle={Surveys in combinatorics 2005},
      series={London Math. Soc. Lecture Note Ser.},
      volume={327},
   publisher={Cambridge Univ. Press},
     address={Cambridge},
       pages={227--258},
      review={\MR{MR2187740}},
}

\bib{MR0115221}{article}{
      author={Harris, T.~E.},
       title={A lower bound for the critical probability in a certain
  percolation process},
        date={1960},
     journal={Proc. Cambridge Philos. Soc.},
      volume={56},
       pages={13--20},
      review={\MR{MR0115221 (22 \#6023)}},
}

\bib{MR1138428}{article}{
      author={Janson, S.},
       title={Poisson approximation for large deviations},
        date={1990},
        ISSN={1042-9832},
     journal={Random Structures \& Algorithms},
      volume={1},
      number={2},
       pages={221--229},
      review={\MR{MR1138428 (93a:60041)}},
}

\bib{MR1782847}{book}{
      author={Janson, S.},
      author={{\L}uczak, T.},
      author={Rucinski, A.},
       title={Random graphs},
      series={Wiley-Interscience Series in Discrete Mathematics and
  Optimization},
   publisher={Wiley-Interscience, New York},
        date={2000},
        ISBN={0-471-17541-2},
      review={\MR{MR1782847 (2001k:05180)}},
}

\bib{MR2096818}{article}{
      author={Janson, Svante},
      author={Ruci{\'n}ski, Andrzej},
       title={The deletion method for upper tail estimates},
        date={2004},
        ISSN={0209-9683},
     journal={Combinatorica},
      volume={24},
      number={4},
       pages={615--640},
         url={http://dx.doi.org/10.1007/s00493-004-0038-3},
      review={\MR{2096818 (2005i:60019)}},
}

\bib{MR1661982}{incollection}{
      author={Kohayakawa, Y.},
       title={Szemer\'edi's regularity lemma for sparse graphs},
        date={1997},
   booktitle={Foundations of computational mathematics (rio de janeiro, 1997)},
   publisher={Springer},
     address={Berlin},
       pages={216--230},
      review={\MR{MR1661982 (99g:05145)}},
}

\bib{MR1609513}{article}{
      author={Kohayakawa, Y.},
      author={Kreuter, B.},
       title={Threshold functions for asymmetric {R}amsey properties involving
  cycles},
        date={1997},
        ISSN={1042-9832},
     journal={Random Structures \& Algorithms},
      volume={11},
      number={3},
       pages={245--276},
      review={\MR{MR1609513 (99g:05159)}},
}

\bib{MR1479298}{article}{
      author={Kohayakawa, Y.},
      author={{\L}uczak, T.},
      author={R{\"o}dl, V.},
       title={On {$K\sp 4$}-free subgraphs of random graphs},
        date={1997},
        ISSN={0209-9683},
     journal={Combinatorica},
      volume={17},
      number={2},
       pages={173--213},
      review={\MR{MR1479298 (98h:05166)}},
}

\bib{MR1980964}{article}{
      author={Kohayakawa, Y.},
      author={R{\"o}dl, V.},
       title={Regular pairs in sparse random graphs.~{I}},
        date={2003},
        ISSN={1042-9832},
     journal={Random Structures Algorithms},
      volume={22},
      number={4},
       pages={359--434},
      review={\MR{MR1980964 (2004b:05187)}},
}

\bib{MR1606853}{article}{
      author={Kreuter, B.},
       title={Threshold functions for asymmetric {R}amsey properties with
  respect to vertex colorings},
        date={1996},
        ISSN={1042-9832},
     journal={Random Structures \&Algorithms},
      volume={9},
      number={3},
       pages={335--348},
      review={\MR{MR1606853 (99c:05178)}},
}

\bib{MR1182457}{article}{
      author={{\L}uczak, T.},
      author={Ruci{\'n}ski, A.},
      author={Voigt, B.},
       title={Ramsey properties of random graphs},
        date={1992},
        ISSN={0095-8956},
     journal={J.~Combin. Theory Ser.~B},
      volume={56},
      number={1},
       pages={55--68},
      review={\MR{MR1182457 (94b:05172)}},
}

\bib{MR2531778}{article}{
      author={Marciniszyn, M.},
      author={Skokan, J.},
      author={Sp{\"o}hel, R.},
      author={Steger, A.},
       title={Asymmetric {R}amsey properties of random graphs involving
  cliques},
        date={2009},
        ISSN={1042-9832},
     journal={Random Structures Algorithms},
      volume={34},
      number={4},
       pages={419--453},
         url={http://dx.doi.org/10.1002/rsa.20239},
      review={\MR{MR2531778 (2010i:05313)}},
}

\bib{MR1249720}{incollection}{
      author={R{\"o}dl, V.},
      author={Ruci{\'n}ski, A.},
       title={Lower bounds on probability thresholds for {R}amsey properties},
        date={1993},
   booktitle={Combinatorics, Paul Erd{\H{o}}s is eighty, vol.~1},
      series={Bolyai Soc. Math. Stud.},
   publisher={J\'anos Bolyai Math. Soc.},
     address={Budapest},
       pages={317--346},
      review={\MR{MR1249720 (95b:05150)}},
}

\bib{MR1276825}{article}{
      author={R{\"o}dl, V.},
      author={Ruci{\'n}ski, A.},
       title={Threshold functions for {R}amsey properties},
        date={1995},
        ISSN={0894-0347},
     journal={J. Amer. Math. Soc.},
      volume={8},
      number={4},
       pages={917--942},
      review={\MR{MR1276825 (96h:05141)}},
}

\bib{MR1492867}{article}{
      author={R{\"o}dl, V.},
      author={Ruci{\'n}ski, A.},
       title={Ramsey properties of random hypergraphs},
        date={1998},
        ISSN={0097-3165},
     journal={J. Combin. Theory Ser. A},
      volume={81},
      number={1},
       pages={1--33},
         url={http://dx.doi.org/10.1006/jcta.1997.2785},
      review={\MR{MR1492867 (98m:05175)}},
}

\bib{MR2318677}{article}{
      author={R{\"o}dl, V.},
      author={Ruci{\'n}ski, A.},
      author={Schacht, M.},
       title={Ramsey properties of random {$k$}-partite, {$k$}-uniform
  hypergraphs},
        date={2007},
        ISSN={0895-4801},
     journal={SIAM J. Discrete Math.},
      volume={21},
      number={2},
       pages={442--460 (electronic)},
         url={http://dx.doi.org/10.1137/060657492},
      review={\MR{MR2318677 (2008d:05103)}},
}

\bib{up:sch}{unpublished}{
      author={Schacht, M.},
       title={Extremal results for random discrete structures},
        note={Submitted},
}

\end{biblist}
\end{bibdiv}
\end{document}